\documentclass[12pt,reqno]{amsart}
\usepackage{latexsym}
\usepackage{amssymb}
\usepackage{verbatim}
\usepackage{version}

\includeversion{comment}

\numberwithin{equation}{section}

\catcode`\@=11
\@namedef{subjclassname@2010}{%
  \textup{2010} Mathematics Subject Classification}
\catcode`\@=12

\newtheorem{theorem}{Theorem}

\newtheorem{proposition}[theorem]{Proposition}
\newtheorem{lemma}[theorem]{Lemma}
\newtheorem{corollary}[theorem]{Corollary}
\newtheorem{conjecture}[theorem]{Conjecture}

\theoremstyle{remark}
\newtheorem*{remark}{Remark}
\newtheorem*{remarks}{Remarks}

\def\fl#1{\left\lfloor#1\right\rfloor}
\def\cl#1{\left\lceil#1\right\rceil}
\def\si{\sigma}

\newcommand{\Z}{\mathbb{Z}}
\newcommand{\Q}{\mathbb{Q}}
\newcommand{\R}{\mathbb{\R}}

\def\al{\alpha}

\def\la{\lambda}
\def\Ga{\Gamma}

\begin{document}

\title[Truncations of Dwork's lemma]{Truncated versions of Dwork's lemma for exponentials
of power series and $p$-divisibility of arithmetic functions}

\author[C. Krattenthaler and 
T.\,W. M\"uller]{C. Krattenthaler$^{\dagger}$ and
T. W. M\"uller$^*$} 

\address{$^{\dagger}$Fakult\"at f\"ur Mathematik, Universit\"at Wien, 
Oskar-Morgenstern-Platz~1, A-1090 Vienna, Austria.
WWW: {\tt http://www.mat.univie.ac.at/\lower0.5ex\hbox{\~{}}kratt}.}

\address{$^*$School of Mathematical Sciences, Queen Mary \& Westfield College, 
University of London,
Mile End Road, London E1 4NS, United Kingdom.}

\thanks{$^\dagger$Research partially supported by the Austrian
Science Foundation FWF, grants Z130-N13 and S50-N15,
the latter in the framework of the Special Research Program
``Algorithmic and Enumerative Combinatorics"\newline\indent
$^*$Research supported by Lise Meitner Grant M1661-N25 of the Austrian
Science Foundation FWF}



\begin{abstract}
(Dieudonn\'e and) Dwork's lemma gives a necessary and sufficient
condition for an exponential of a formal power series $S(z)$ with
coefficients in $\mathbb Q_p$ to have coefficients in $\mathbb Z_p$. 
We establish theorems on the $p$-adic valuation of the coefficients
of the exponential of $S(z)$, assuming weaker conditions on the
coefficients of $S(z)$ than in Dwork's lemma. As applications,
we provide several results concerning lower bounds on the
$p$-adic valuation of the number of permutation representations
of finitely generated groups. In particular, we give fairly tight
lower bounds in the case of an arbitrary finite Abelian $p$-group,
thus generalising numerous results in special cases that had appeared
earlier in the literature. Further applications include sufficient
conditions for ultimate periodicity of subgroup numbers modulo~$p$
for free products of finite Abelian $p$-groups, results on
$p$-divisibility of permutation numbers with restrictions on their
cycle structure, and a curious
``supercongruence" for a certain binomial sum.  
\end{abstract}

\maketitle

\section{Introduction}

\noindent For a finitely generated group $\Gamma$ and a positive integer $n$,
let $s_n(\Gamma)$ denote the number of subgroups of index $n$ in
$\Gamma$. Moreover, for a non-negative integer $n$, set  
\[
h_n(\Gamma)= 
\vert\mathrm{Hom}(\Gamma, S_n)\vert,
\]
the number of representations of $\Gamma$ in the symmetric group~$S_n$.
It is well-known that the sequences $\big(s_n(\Gamma)\big)_{n\geq1}$ and
$\big(h_n(\Gamma)\big)_{n\geq0}$ are related via the (formal) identity 
\begin{equation} 
\label{eq:HS}
\sum_{n\ge0} \frac {h_n(\Gamma)} {n!} z^n = \exp\left(
\sum_{n\ge1} \frac {s_n(\Gamma)} {n} z^n\right); 
\end{equation}
cf., for instance, \cite[Prop.~1]{DM}. By differentiation of both sides, 
Equation~\eqref{eq:HS} is seen to be equivalent to the recurrence relation
\begin{equation}
\label{Eq:HallTransform}
h_n(\Gamma) = \sum_{k=1}^n ( n-k+1)_{k-1}\, s_k(\Gamma)\,
h_{n-k}(\Gamma),\quad \quad n\geq1, 
\end{equation}
where, for
$\alpha\in\mathbb{C}$ and $n\in\mathbb{N}_0$,  
\[
(\alpha)_n = \begin{cases}
\alpha(\alpha+1)\cdots(\alpha+n-1),& n>0,\\
1,& n=0,\end{cases}
\]
is the Pochhammer symbol. 

For a prime number $p$ and a
  positive integer $n$, the $p$-adic valuation $v_p(n)$ is defined as the
  exponent of the largest power of $p$ dividing $n$.
At the origin of the research presented here stands
the aim of deriving good lower
bounds for the $p$-adic valuation 
$v_p\big(h_n(G)\big)$, where $p$ is a prime and $G$ is a finite Abelian
$p$-group, which are sharp for infinitely many $n$. As a matter of fact,
our main results turn out to be much more general than that.  

It would appear that investigation into the above problem started in 1951
with the paper \cite{ChHMAA} by Chowla, Herstein, and Moore, which,
among other things, contains the result that\footnote{For an integer
  $m\geq2$, we denote by $C_m$ the cyclic group of order $m$.}  
\begin{equation} \label{eq:hnC2} 
v_2\big(h_n(C_2)\big) \geq \fl{\frac{n+2}{4}} = 
\fl{\frac{n}{2}} - \fl{ \frac{n}{4}};  
\end{equation}
cf.\ \cite[Theorem~10]{ChHMAA}. The more general result to the effect that 
\begin{equation}
\label{Eq:CpVal}
v_p\big(h_n(C_p)\big) \geq \fl{ \frac{n}{p}} - 
\fl{\frac{n}{p^2}},\quad n\geq1, 
\end{equation}
for arbitrary prime numbers $p$, 
was established several times by various sets of
authors, apparently starting with Dress and Yoshida 
 \cite{DrYoAA} in 1991, followed by Grady and Newman
 \cite{GrNeAA}, and by Ochiai \cite[Sec.~3.2]{Ochiai}, among others. 
Ochiai's paper also contains somewhat sharper results for small
 primes; for instance, he shows that  
\[
v_2\big(h_n(C_2)\big) = 
(n+5)/4, \quad \text{for }n\equiv 3\pmod{4},
\]
which is an improvement over (\ref{Eq:CpVal}) in that case. 
The most general results in this direction, prior to the present
paper, are due to Katsurada, Takegahara, and  Yoshida, who establish
lower bounds for the $p$-adic valuation $v_p\big(h_n(G)\big)$ in the case when
$G$ is a finite Abelian $p$-group of rank (minimal number of
generators) at most $2$;
cf. Theorems~1.3 and 1.4 in \cite{KaTYAA}. For instance, they show
that, for a prime $p$ and integers $\ell, m$ with $\ell \geq m\geq0$, 
\[
v_p\big(h_n(C_{p^\ell} \times C_{p^m})\big) \geq
\sum\limits_{j=1}^\ell\fl{\frac{n}{p^j}} - (\ell - m)
\fl{\frac{n}{p^{\ell+1}}}.
\]
In the present paper, we start from an arbitrary sequence $s_1, s_2,
s_3, \ldots\in \Q_p$ and define a sequence $h_0, h_1,
h_2,\ldots\in\Q_p$ by means of the identity 
\[
H(z):= \sum_{n\geq0} \frac {h_n} {n!} z^n = \exp\left(S(z)\right),
\]
where $S(z):= \sum_{n\geq1} \frac{s_n}{n} z^n$ is the logarithmic
generating function for the sequence $(s_n)_{n\geq1}$; in other words,
the sequences $(s_n)_{n\geq1}$ and
$(h_n)_{n\geq0}$ are related via (an abstract version of) the
recurrence relation (\ref{Eq:HallTransform}). Our main idea is to
establish \textit{truncated versions} of an integrality criterion
essentially due to Dwork, sometimes also attributed to Dieudonn\'e and
Dwork (cf. \cite[Ch.~14, p.~76]{LangAA} and the theorem on p. 409 of
\cite{RobeAA} with $A = \Z_p$ and $B = \Q_p$). We begin by stating
this integrality criterion in the precise form needed here. 

\begin{lemma} \label{lem:dwork}
For a prime number $p,$
let\/ $S(z)$ and $H(z)$ be formal power series with coefficients in
$\Q_p$ related by
\begin{equation} \label{eq:H-S}
H(z)=\exp\big(S(z)\big). 
\end{equation}
Then $H(z)$ has all coefficients in $\Z_p$ if, and only if,
\begin{equation} \label{eq:pS1d}
S(z^{p})-pS(z)\in p\Z_p[[z]]. 
\end{equation}
\end{lemma}
\noindent Our results do not attempt to provide necessary and
sufficient conditions  
as in Lemma~\ref{lem:dwork} (that would perhaps seem overly ambitious);
rather, they provide {\it sufficient\/} conditions 
for a {\it certain} $p$-divisibility. Thus, we obtain sufficient
conditions for an exponential of a power series to have coefficients 
which are divisible by certain scales of powers of a given prime
$p$. Which scales one should think of, we gleaned off the scales one
finds in Theorems~1.2--1.4 of \cite{KaTYAA}, which, from our point of
view, provide $p$-divisibility results for the ``simplest"
cases. Theorems~\ref{thm:pdiv1} and \ref{thm:pdiv12} in Section
\ref{sec:pdiv1} vastly generalise these results. We would like to
point out though that the proofs of our theorems draw on a key idea in
the proof of Theorem~1.2 in \cite{KaTYAA}: the idea of comparison with
a reference power series. 

In our first main result, Theorem~\ref{thm:pdiv1},
we ``truncate" Condition~\eqref{eq:pS1d} in
Lemma~\ref{lem:dwork} in the sense that  
we require \eqref{eq:pS1d} to hold only up to (and including) the
coefficient of $z^{p^l-1}$, for some given~$l$ (see \eqref{eq:pS1}),
assuming only certain weaker conditions for coefficients of higher
powers  
(see \eqref{eq:p^m} and \eqref{eq:sdiff}). This has the effect of
replacing $p$-integrality of the coefficients of the series $H(z)$ in
\eqref{eq:H-S} by the weaker $p$-divisibility result for the
coefficients given in \eqref{eq:vpH1}. (These results have to be
compared with the bound
\begin{equation} \label{eq:hnn!} 
v_p(h_n)\ge v_p(n!)=\sum_{s\geq1} \fl{\frac{n}{p^s}}
\end{equation}
implied by Lemma~\ref{lem:dwork} provided that \eqref{eq:pS1d} holds,
the equality on the right being due to 
Legendre's formula \cite[p.~10]{LegeAA} for the $p$-adic
valuation of $n!$.)

As it turns out, for $p=2$, there exists a case not covered by
Theorem~\ref{thm:pdiv1} (since Condition~(\ref{eq:sdiff}) is violated
by a small margin), where one may nevertheless even \textit{improve}
on the bound given in (\ref{eq:vpH1}). This is the topic of
Theorem~\ref{thm:pdiv12}, which vastly generalises Theorem~1.4 of
\cite{KaTYAA}.  
Section~\ref{sec:pdiv1} is devoted to the proofs of
Theorems~\ref{thm:pdiv1} and \ref{thm:pdiv12}, while also providing
considerable further comment as well as several corollaries. 

Section~\ref{sec:pdiv2} addresses the case where we have less precise
information on the coefficients $s_n$ of the series $S(z)$ in
\eqref{eq:H-S}. More specifically, what can we say if  
Condition \eqref{eq:pS1} is ``truncated" to hold only up to the
coefficient of $z^{p^{l}}$ for some $l$, while beyond we do not have
any additional information, except that $s_n\in \Z_p$? While we cannot
expect a result as strong as, say, the bound \eqref{eq:vpH10} 
in Corollary~\ref{thm:pdiv10} with $l$ replaced by $l+1$, 
$p$-divisibility of the coefficients of the exponential $H(z)$ in
\eqref{eq:H-S} still turns out to be higher than one might expect in
that situation. Theorems~\ref{thm:pdiv2}, \ref{thm:pdiv31},
\ref{thm:pdiv22}, and \ref{thm:pdiv3}  of Section~\ref{sec:pdiv2}
collectively address this problem. For instance,
Theorem~\ref{thm:pdiv2} provides the bound 
\[
v_p(h_n) \geq \sum_{s\geq1} \fl{\frac{n}{p^s}}
-(l-1)\fl{\frac{n}{p^l}}- \sum_{s\geq l}
\fl{\frac{n}{2p^s}},\quad n\geq1, 
\]
valid for $p\geq3$ and $(p, l)\neq (3,1)$. In particular,
Theorems~\ref{thm:pdiv2}, \ref{thm:pdiv31}, and \ref{thm:pdiv22},
together with the special case of Theorem~\ref{thm:pdiv1} where $l=1$
and $m=0$, combine to demonstrate existence of a sharp
\textit{dividing line} concerning the $p$-divisibility of the coefficient
$h_n$: if $s_1\not\equiv s_p\mod{p\Z_p}$ then, for each $N$,
there exists some $n>N$, such that $h_n$ is not divisible by $p$; if,
on the other hand, $s_1 \equiv s_p\mod{p\Z_p}$, then $v_p(h_n)
\rightarrow\infty$ as $n$ tends to infinity;
cf.\ Corollary~\ref{cor:divline} for more precise statements.  

The rest of our paper is devoted to applications of our abstract 
$p$-divisibility results in Sections~\ref{sec:pdiv1} and \ref{sec:pdiv2}. 
Sections~\ref{sec:fingen} and \ref{sec:appl} 
address the question what may be said concerning the
$p$-divisibility of the homomorphism numbers $h_n(\Gamma)$ where
$\Gamma$ is a finitely generated group. 
In Proposition~\ref{Prop:FinGenIndexp}, it is shown that
$s_p(\Gamma)\equiv 0,1$~(mod~$p$). This allows us to rephrase
the dividing line concerning $p$-divisibility outlined above 
in this specific case as follows: if 
$s_p(\Gamma)\equiv 1$~(mod~$p$), then $v_p\big(h_n(\Gamma)\big)$
tends to infinity as $n\to\infty$, whereas for
$s_p(\Gamma)\equiv 0$~(mod~$p$) the prime~$p$ does not divide 
$h_n(\Gamma)$ for infinitely many~$n$; see Theorem~\ref{thm:fingen}.
Section~\ref{sec:appl} also contains further $p$-divisibility results
for homomorphism numbers $h_n(G)$, where $G$ is a finite $p$-group
(see Theorem~\ref{thm:pgroup}) or a finite non-Abelian simple group (see
Theorem~\ref{thm:Kul}). 

Our main applications are contained in Section~\ref{sec:Abel}. They
provide tight lower bounds for the $p$-divisibility of
homomorphism numbers of arbitrary finite Abelian $p$-groups.
Theorems~\ref{thm:rang} and \ref{thm:rang2} generalise
the earlier results of Katsurada, Takegahara and Yoshida, which
concerned the case of Abelian $p$-groups of rank at most~$2$, to
arbitrary rank.

We conclude our paper with some further applications.
The results in Section~\ref{sec:per} provide sufficient conditions
for ultimate periodicity modulo~$p$ 
of the subgroup numbers $s_n(\Gamma)$, where $p$ is a prime number and
$\Gamma$ is a free product of finite Abelian $p$-groups.
The corresponding theorem, Theorem~\ref{thm:per},
generalises earlier findings by Grady and Newman \cite{GrNeAA}.
In Section~\ref{sec:perm}, we apply our abstract $p$-divisibility
results to numbers of permutations with restrictions on their
cycle structure. Finally, Section~\ref{sec:super} presents a
curious three-parameter congruence modulo higher powers of
a prime~$p$.

\section{$p$-Divisibility of coefficients in exponentials of
  power series, I} \label{sec:pdiv1}

\noindent In this section we present our first set of 
main results. They concern the case where Condition~\eqref{eq:pS1d} in
Lemma~\ref{lem:dwork} is truncated after the term
involving $z^{p^l-1}$, 
for some given~$l$ (see \eqref{eq:pS1}), while assuming weaker
conditions 
for coefficients of higher powers  
(see \eqref{eq:p^m} and \eqref{eq:sdiff}). This results in 
replacing $p$-integrality of the coefficients of the series $H(z)$ in
\eqref{eq:H-S} by the weaker $p$-divisibility result given in
\eqref{eq:vpH1}. The proof of Theorem~\ref{thm:pdiv1} 
requires an auxiliary result, which is treated separately in
Lemma~\ref{lem:ij}.

\begin{theorem} \label{thm:pdiv1}
For a prime number $p,$
let\/ $S(z)=\sum_{n\ge1}\frac {s_n} {n}z^n$ be a formal power series
with $s_n\in\Q_p$ for all $n,$ and let 
$H(z)=\sum_{n\ge0}\frac {h_n} {n!}z^n$ be the exponential of $S(z)$.
Given non-negative integers $l$ and $m$ with $m<l,$ we assume that
\begin{equation} \label{eq:pS1}
S(z^{p})-pS(z)=pJ(z)+(s_{p^{l-1}}-s_{p^l})\frac {z^{p^l}} {p^{l-1}}
+O\left(z^{p^{l}+1}\right) 
\end{equation}
with $J(z)\in \Z_p[z],$ that
\begin{equation} \label{eq:p^m} 
s_{p^{l-1}}\equiv s_{p^{l}}\quad 
\text{\em mod~$p^{m}\Z_p$},
\end{equation}
and that
\begin{equation} \label{eq:sdiff} 
v_p(\la_{i})\ge -(l-m)\fl{\frac {i} {p^{l}}}
%-\sum_{s\ge l}\fl{\frac {i} {p^{s}}}
+v_p(i)-\frac {p^{\fl{\log_pi}-l}-1} {p-1}+1
%+\chi(i<p^{l+1})
\end{equation}
for all $i>p^{l},$ where $\la_i=s_i$ if $i/p^{v_p(i)}\ge p^l$
and $\la_i=s_i-s_{i/p^{e}}$ otherwise, where $e$ is minimal such
that $i/p^e< p^l$. 
%Here, $\chi(\mathcal A)=1$
%if $\mathcal A$ is true, and $\chi(\mathcal A)=0$ otherwise.
Then
\begin{equation} \label{eq:vpH1}
v_p(h_n)\ge \sum_{s=1}^{l-1} \fl{\frac{n}{p^s}}
-(l-m-1)\fl{\frac {n} {p^l}}
\end{equation}
for all $n$.

Moreover, writing 
$e_p(n;l,m)=\sum_{s=1}^{l-1}\big\lfloor\frac {n} {p^{s}}\big\rfloor
-(l-m-1)\big\lfloor\frac {n} {p^l}\big\rfloor,$ 
under these conditions, the quotient
\begin{equation} \label{eq:quot}
Q_n=\frac {h_n} 
{p^{e_p(n;l,m)}}
\end{equation}
satisfies
\begin{equation} \label{eq:Qrek1} 
Q_n\equiv (-1)^lp^{-m}(s_{p^l}-s_{p^{l-1}})Q_{n-p^l} 
\quad \text{\em mod $p\Z_p$}.
\end{equation}
In particular, if 
\hbox{$s_{p^{l-1}}\not\equiv s_{p^{l}}$~{\em mod~$p^{m+1}\Z_p,$}}
then the bound in \eqref{eq:vpH1} is tight for all $n$ which are
divisible by $p^l$.
\end{theorem}

\begin{remarks} \label{rem:1}
(1) In order to compare the above truncated version with the original
Lem\-ma~\ref{lem:dwork}, one must compare
the bound in \eqref{eq:vpH1} with $v_p(n!)$ (recall the
parenthetical remark containing \eqref{eq:hnn!}).
More specifically, what we loose by truncating \eqref{eq:pS1d} to
\eqref{eq:pS1}--\eqref{eq:sdiff} is the difference
$$
\sum_{s\geq1} \fl{\frac{n}{p^s}}
-
\left(\sum_{s=1}^{l-1} \fl{\frac{n}{p^s}}
-(l-m-1)\fl{\frac {n} {p^l}}\right)
=(l-m)\fl{\frac {n} {p^l}}+\sum_{s\ge l+1} \fl{\frac{n}{p^s}}.
$$
%(1) It should be observed that the assertion
%that $Q_n$ is an element of $\Z_p$ 
%is equivalent to the bound in \eqref{eq:vpH1}, since
%\begin{align*}
%v_p\left(\frac {n!\,H_n} 
%{p^{e_p(n;l,m)}}\right)&=
%v_p(n!)+v_p(H_n)-
%\sum_{s=1}^{l-1}\fl{\frac {n} {p^{s}}}
%+(l-m-1)\fl{\frac {n} {p^l}}\\
%&=
%v_p(H_n)+(l-m)\fl{\frac {n} {p^l}}
%+\sum_{s\ge l+1}\fl{\frac {n} {p^{s}}},
%\end{align*}
%where we have used Legendre‘s formula \cite[p.~10]{LegeAA} in the last line.
%
%The ``normalisation" by $n!$ in \eqref{eq:quot} may look somewhat
%artificial; 
%however, it is done for two good reasons: first, in many!!all applications
%it is the quantity $h_n:=n!\, H_n$ that we are actually interested
%in (see Section~\ref{sec:Abel}), and, second, 
%it has the advantage that the 
%periodicity statement \eqref{eq:Qrek1} becomes more elegant
%than without this factor.

\medskip
(2) The degree of the polynomial $J(z)$ may be restricted to 
$p^l-1$.

\medskip
(3) The term $(s_{p^{l-1}}-s_{p^l})\frac {z^{p^l}} {p^{l-1}}$ in
\eqref{eq:pS1} occurs by the definition of $S(z)$; it does not
impose any condition. The condition on that term is formulated
in \eqref{eq:p^m}.

\medskip
(4) It is worth singling out the special case of Theorem~\ref{thm:pdiv1}
where $l=1$ and $m=0$: if 
\hbox{$s_{1}\not\equiv s_{p}$~mod~$p\Z_p$}, then 
$h_n\equiv -(s_{p}-s_{1})\,h_{n-p} $~mod~$p\Z_p$. Consequently,
$$h_{ap}\equiv(s_{1}-s_{p})^a~\text{mod}~p\Z_p,\quad a\geq0;$$ 
thus, for each given $N$, there exists some 
$n>N$ such that $h_n$ is not divisible by $p$.
On the other hand, if
$s_{1}\equiv s_{p}$~mod~$p\Z_p$, then it follows from 
Theorems~\ref{thm:pdiv2}, \ref{thm:pdiv31}, and \ref{thm:pdiv22} 
with $l=1$
that larger and larger
powers of $p$ divide $h_n$ as $n$ grows; see Corollary~\ref{cor:divline}.
This says that there is a sharp dividing line concerning the
$p$-divisibility of $h_n$ depending on whether
$s_{1}\equiv s_{p}$~mod~$p\Z_p$ or not.
\end{remarks}

\begin{proof}[Proof of Theorem~{\em \ref{thm:pdiv1}}]
%{\sc The ``only if" part}. 
We write $n=ap^l+r$ with $0\le r<p^l$.
Assuming \eqref{eq:pS1}--\eqref{eq:sdiff},
we have to show the bound \eqref{eq:vpH1}.

We rewrite the series $S(z)$ as
\begin{equation} \label{eq:H-H2} 
S(z)=\widetilde S(z)+(s_{p^l}-s_{p^{l-1}})\frac {z^{p^l}}
{p^l} + \sum_{i=p^l+1}^{\infty}\la_i\frac {z^{i}} {i},
\end{equation}
with the coefficients $\la_i\in \Z_p$ as in the statement
of the theorem. It should be observed that this makes 
$\widetilde S(z)$ satisfy \eqref{eq:pS1d}. 

Let us write $\widetilde
H(z)=\sum_{n\ge0}\widetilde H_nz^n=\exp\big(\widetilde S(z)\big)$.
By taking the exponential of both sides of \eqref{eq:H-H2}, we obtain 
\begin{equation} \label{eq:H-H}
H(z) =
\widetilde H(z)\exp\left((s_{p^l}-s_{p^{l-1}})\frac {z^{p^l}}
{p^l} + \sum_{i=p^l+1}^{\infty}\la_i\frac {z^{i}} {i}
\right) .
\end{equation}
We compare coefficients of $z^{ap^l+r}$ on both sides of
\eqref{eq:H-H}. This leads to
\begin{equation} \label{eq:H-H3} 
\frac {h_{ap^l+r}} {(ap^l+r)!}=
\sum_{i_0+p^lj_0+\sum_{i=p^l+1}^\infty ij_i=ap^l+r}
\widetilde H_{i_0}
\frac {(s_{p^l}-s_{p^{l-1}})^{j_0}} {p^{j_0l}\,j_0!}
\prod _{i=p^l+1} ^{\infty}
\frac {\la_i^{j_i}}
{i^{j_i}\,j_i!}.
\end{equation}
We now bound the $p$-adic valuation of the summand from below.
By Lemma~\ref{lem:dwork} applied to $\widetilde S(z)$, 
we know that $\widetilde H_{i_0}\in \Z_p$.
By \eqref{eq:p^m}, we have
\begin{equation*} %\label{eq:j_0}
v_p\left(\frac {(s_{p^l}-s_{p^{l-1}})^{j_0}} {p^{j_0l}\,j_0!}
\right) 
\ge
j_0(m-l)-\sum_{s\ge1}\fl{\frac {j_0} {p^s}}.
\end{equation*}
By \eqref{eq:sdiff}, for $i>p^l$ we have
\begin{equation*} %\label{eq:j_i1}
v_p\left(\frac {\la_{i}^{j_{i}}}
{i^{j_{i}}\,j_{i}!}\right) 
\ge 
-j_{i}\left((l-m)\fl{\frac {i} {p^{l}}}
+\frac {p^{\fl{\log_pi}-l}-1} {p-1}-1
%-\chi(i<p^{l+1})
\right)
%-\sum_{s\ge l+1}\fl{\frac {i} {p^{s}}}
-\sum_{s\ge1}\fl{\frac {j_{i}} {p^s}}.
\end{equation*}
Together, these estimations imply that the $p$-adic valuation of the
summand of the sum on the right-hand side of \eqref{eq:H-H3} is at least
{\allowdisplaybreaks
\begin{align}
\notag
-(l-m)&\left(j_0+\sum_{i=p^l+1}^\infty j_i\fl{\frac {i} {p^{l}}} \right)
+\sum_{i=p^{l}+1}^{\infty} j_{i}
-\sum_{i=p^{l}+1}^{\infty} \frac {p^{\fl{\log_pi}-l}-1} {p-1}j_{i}\\
\notag
&\kern5cm
-\sum_{s\ge1}\left(\fl{\frac {j_0} {p^s}}
+\sum_{i=p^l+1}^\infty\fl{\frac {j_i} {p^s}}\right)\\
\notag
&\kern1cm
\ge
-(l-m)\Bigg\lfloor j_0+\sum_{i=p^l+1}^\infty j_i\frac {i} {p^{l}} \Bigg\rfloor
-\sum_{i=p^{l}+1}^{\infty} \frac {p^{\fl{\log_pi}-l}-1} {p-1}j_{i}\\
\notag
&\kern2cm
-\sum_{s\ge1}\left(\fl{\frac {j_0} {p^s}}
+\sum_{i=p^l+1}^\infty\fl{\frac {ij_i} {p^{s+l}}}\right)
+\sum_{s\ge1}
\sum_{i=p^l+1}^\infty\left(\fl{\frac {ij_i} {p^{s+l}}}
-\fl{\frac {j_i} {p^{s}}}
\right)\\
\notag
&\kern1cm
\ge
-(l-m)\Bigg\lfloor{\frac {1} {p^l}
\Bigg(p^lj_0+\sum_{i=p^l+1}^\infty ij_i\Bigg)}\Bigg\rfloor
-\sum_{s\ge1}\fl{\frac {p^lj_0+\sum_{i=p^l+1}^\infty ij_i} {p^{s+l}}}\\
&\kern1cm
\ge
-(l-m)a
-\sum_{s\ge1}\fl{\frac {ap^l+r} {p^{s+l}}},
\label{eq:bound}
\end{align}}%
where we have used Lemma~\ref{lem:ij} to obtain the next-to-last line, and
the dependence of the summation indices of the sum
in \eqref{eq:H-H3} to obtain the last line.
Together with Legendre's formula \cite[p.~10]{LegeAA} applied to the
$p$-adic valuation of $(ap^l+r)!$ on the left-hand side of \eqref{eq:H-H3}, 
this proves the bound in
\eqref{eq:vpH1}, with $n=ap^l+r$.

In order to establish the more precise congruence
\eqref{eq:Qrek1}, we first observe that equality in \eqref{eq:bound}
only holds if $j_i=0$ for all $i>p^{l}$ since we discarded
the sum of these $j_i$'s when going from the first to the second
expression.
Hence, it suffices to look more carefully at
the summand with $i_0=r$, $j_0=a$, and $j_i=0$ for $i>p^l$,
%Namely, by Lemma~\ref{lem:dwork} applied to $\widetilde S(z)$, 
%we know that $\widetilde H_r\in \Z_p$
that is, at
$$
\widetilde H_{r}
\frac {(s_{p^l}-s_{p^{l-1}})^{a}} {p^{a l}\,a!}.
$$
More precisely, by the above considerations, we conclude that
\begin{align*}
Q_{ap^l+r}&=\frac {h_{ap^l+r}} 
{p^{e_p(ap^l+r;m,l)}}\\
&\equiv
\frac {(ap^l+r)!\,\widetilde H_{r}} {p^{e_p(ap^l+r;m,l)}}
\frac {(s_{p^l}-s_{p^{l-1}})^{a}} {p^{a l}\,a!}
\quad \text{mod }p\Z_p\\
&\equiv
\frac {(ap^l+r)!\,\widetilde H_{r}} {p^{e_p(r;m,l)+ap^{l-1}+\dots+ap+a}}
\frac {\left(p^{-m}(s_{p^l}-s_{p^{l-1}})\right)^{a}} {a!}
\quad \text{mod }p\Z_p\\
&\equiv
\frac {(p-1)!^{a(p^{l-1}+\dots+p+1)}\,r!\,\widetilde H_{r}} 
{p^{e_p(r;m,l)}}
{\left(p^{-m}(s_{p^l}-s_{p^{l-1}})\right)^{a}} 
\quad \text{mod }p\Z_p\\
&\equiv
Q_r
(-1)^{al}
{\left(p^{-m}(s_{p^l}-s_{p^{l-1}})\right)^{a}} 
\quad \text{mod }p\Z_p,
\end{align*}
where we have used Wilson's theorem plus the fact 
that $\widetilde H_r=\frac {h_r} {r!}$ in the last line.
This is exactly the congruence \eqref{eq:Qrek1}.

The above computation with $r=0$ implies that
\begin{equation*} %\label{eq:Qrek1} 
Q_{ap^l}\equiv \big((-1)^lp^{-m}(s_{p^l}-s_{p^{l-1}})\big)^a
\quad \text{mod $p\Z_p$}
\end{equation*}
for all positive integers $a$. Thus, if 
\hbox{$s_{p^{l-1}}\not\equiv s_{p^{l}}$~{mod~$p^{m+1}\Z_p$}},
the bound \eqref{eq:vpH1} is
sharp for all $n$ divisible by $p^l$.
This completes the proof of the theorem.
\end{proof}

\begin{lemma} \label{lem:ij}
For all non-negative integers $i,$ $j,$ $l,$ and\/ $p,$ with $i\ge p^{l}$
and $p\ge2,$ we have
$$
\sum_{s\ge1}
\left(\fl{\frac {ij} {p^{s+l}}}
-
\fl{\frac {j} {p^{s}}}\right)
\ge \frac {p^{\fl{\log_pi}-l}-1} {p-1}j.
$$
\end{lemma}

\begin{proof}
We have
\begin{align*}
\sum_{s\ge1}
\left(\fl{\frac {ij} {p^{s+l}}}
-
\fl{\frac {j} {p^{s}}}\right)
&\ge
\sum_{s\ge1}
\left(\fl{\frac {p^{\fl{\log_pi}}j} {p^{s+l}}}
-
\fl{\frac {j} {p^{s}}}\right)\\
&\ge
\sum_{s=1} ^{\fl{\log_pi}-l}
 {p^{\fl{\log_pi}-s-l}}j
=\frac {p^{\fl{\log_pi}-l}-1} {p-1}j,
\end{align*}
as desired.
\end{proof}

In the case where $m=0$, one can formulate a simplified version of
Theorem~\ref{thm:pdiv1} if one assumes in addition that the
coefficients $s_n$ are elements of $\Z_p$, since then \eqref{eq:sdiff}
is automatically satisfied.

\begin{corollary} \label{thm:pdiv10}
For a prime number $p,$
let\/ $S(z)=\sum_{n\ge1}\frac {s_n} {n}z^n$ be a formal power series
with $s_n\in\Z_p$ for all $n,$ and let 
$H(z)=\sum_{n\ge0}\frac {h_n} {n!}z^n$ be the exponential of $S(z)$.
Given a positive integer $l,$ we assume that
\begin{equation} \label{eq:pS10}
S(z^{p})-pS(z)=pJ(z)
+O\left(z^{p^{l}}\right) 
\end{equation}
with $J(z)\in \Z_p[z]$.
Then
\begin{equation} \label{eq:vpH10}
v_p(h_n)\ge \sum_{s=1}^{l-1} \fl{\frac{n}{p^s}}
-(l-1)\fl{\frac {n} {p^l}}
\end{equation}
for all $n$.

Moreover, writing 
$e_p(n;l)=\sum_{s=1}^{l-1}\fl{\frac {n} {p^{s}}}
-(l-1)\fl{\frac {n} {p^l}},$ 
under these conditions, the quotient
\begin{equation*} %\label{eq:quot}
Q_n=\frac {h_n} 
{p^{e_p(n;l)}}
\end{equation*}
satisfies
\begin{equation} \label{eq:Qrek10} 
Q_n\equiv (-1)^l(s_{p^l}-s_{p^{l-1}})Q_{n-p^l} 
\quad \text{\em mod $p\Z_p$}.
\end{equation}
In particular, if 
\hbox{$s_{p^{l-1}}\not\equiv s_{p^{l}}$~{\em mod~$p\Z_p,$}}
then the bound in \eqref{eq:vpH10} is tight for all $n$ which are
divisible by $p^l$.
\end{corollary}

We state another special case of Theorem~\ref{thm:pdiv1} separately,
which is convenient in many applications. More precisely, it addresses
the case where the coefficients $s_n$ may be only non-zero for $n$ a
power of $p$. This is the case, for instance, if the
coefficient $s_n$ is the number of subgroups of index $n$ in a
given $p$-group, $n\ge0$; see Sections~\ref{sec:appl} and \ref{sec:Abel}.

\begin{corollary} \label{thm:pdiv1u}
For a prime number $p,$
let\/ $S(z)=\sum_{n\ge1}\frac {s_n} {n}z^n$ be a formal power series
with $s_{p^e}\in\Z_p$ for all non-negative integers $e$ 
and $s_n=0$ otherwise, and let 
$H(z)=\sum_{n\ge0}\frac {h_n} {n!}z^n$ be the exponential of $S(z)$.
Given non-negative integers $l$ and $m$ with $m<l,$ we assume that
\begin{equation} \label{eq:pS1u}
S(z^{p})-pS(z)=pJ(z)+(s_{p^{l-1}}-s_{p^l})\frac {z^{p^l}} {p^{l-1}}
+O\left(z^{p^{l}+1}\right) 
\end{equation}
with $J(z)\in \Z_p[z],$ 
\begin{equation} \label{eq:p^mu} 
s_{p^{l-1}}\equiv s_{p^{l}}\quad 
\text{\em mod~$p^{m}\Z_p$},
\end{equation}
and
\begin{equation} \label{eq:sdiffu} 
v_p(s_{p^e}-s_{p^{l-1}})\ge -(l-m)p^{e-l}-\frac {p^{e-l}-1} {p-1}+e+1
\end{equation}
for all $e$ with $l<e<l+\log_p(2l+1)$.
Then
\begin{equation} \label{eq:vpH1u}
v_p(h_n)\ge \sum_{s=1}^{l-1} \fl{\frac{n}{p^s}}
-(l-m-1)\fl{\frac {n} {p^l}}
\end{equation}
for all $n$.

Moreover, writing 
$e_p(n;l,m)=\sum_{s=1}^{l-1}\fl{\frac {n} {p^{s}}}
-(l-m-1)\fl{\frac {n} {p^l}},$ 
under these conditions the quotient
\begin{equation*} %\label{eq:quot}
Q_n=\frac {h_n} 
{p^{e_p(n;l,m)}}
\end{equation*}
satisfies
\begin{equation} \label{eq:Qrek1u} 
Q_n\equiv (-1)^lp^{-m}(s_{p^l}-s_{p^{l-1}})Q_{n-p^l} 
\quad \text{\em mod $p\Z_p$}.
\end{equation}
In particular, if 
\hbox{$s_{p^{l-1}}\not\equiv s_{p^{l}}$~{\em mod~$p^{m+1}\Z_p,$}}
then the bound in \eqref{eq:vpH1u} is tight for all $n$ which are
divisible by $p^l$.
\end{corollary}

If $p=2$, there is a case which does not fall under
the conditions of Theorem~\ref{thm:pdiv1} (more precisely, 
Condition~\eqref{eq:sdiff} is violated by a small margin), in which one may
nevertheless even {\it improve} the bound in \eqref{eq:vpH1}. The corresponding
theorem, given below, vastly generalises Theorem~1.4 in \cite{KaTYAA}.
We shall apply it in the proof of Theorem~\ref{thm:rang2}. 

\begin{theorem} \label{thm:pdiv12}
Let\/ $S(z)=\sum_{n\ge1}\frac {s_n} {n}z^n$ be a formal power series
with $s_n\in\Q_2$ for all $n,$ and let 
$H(z)=\sum_{n\ge0}\frac {h_n} {n!}z^n$ be the exponential of $S(z)$.
Given an integer $l\ge2,$ we assume that
\begin{equation} \label{eq:pS12}
S(z^{2})-2S(z)=2J(z)+(s_{2^{l-1}}-s_{2^l})\frac {z^{2^l}} {2^{l-1}}
+O\left(z^{2^{l}+1}\right) 
\end{equation}
with $J(z)\in \Z_2[z],$ that
\begin{gather} \label{eq:p^m2a} 
s_{2^{l-1}}\equiv s_{2^{l}}\quad 
\text{\em mod~$2^{l-1}\Z_2$},
\quad 
s_{2^{l}}\equiv s_{2^{l+1}}\quad 
\text{\em mod~$2^{l-2}\Z_2$},
\\
\label{eq:p^m2b} 
s_{2^{l}}-s_{2^{l-1}}\equiv 2(s_{2^{l+1}}-s_{2^{l}})\quad 
\text{\em mod~$2^{l}\Z_2$},
\end{gather}
and that
\begin{equation} \label{eq:sdiff2} 
v_2(\la_{i})\ge -\fl{\frac {i} {2^{l}}}
%-\sum_{s\ge l}\fl{\frac {i} {2^{s}}}
+v_2(i)-2^{\fl{\log_2i}-l}+\cl{\frac {i} {2^{l+2}}}+1
+\chi(i>2^{l+1})
\end{equation}
for all $i>2^{l}$ different from $2^{l+1},$ 
where $\la_i=s_i$ if $i/2^{v_2(i)}\ge 2^l$
and $\la_i=s_i-s_{i/2^{e}}$ otherwise, and with $e$ minimal such
that $i/2^e< 2^l$. 
Here, $\chi(\mathcal A)=1$
if $\mathcal A$ is true, and $\chi(\mathcal A)=0$ otherwise.
Then
\begin{equation} \label{eq:vpH12}
v_2(h_n)\ge \sum_{s=1}^{l-1} \fl{\frac{n}{2^s}}
+\fl{\frac {n} {2^{l+1}}}-\fl{\frac {n} {2^{l+2}}}
\end{equation}
for all $n$.

Moreover, writing 
$e_2(n;l)=\sum_{s=1}^{l-1}\fl{\frac {n} {2^{s}}}
+\fl{\frac {n} {2^{l+1}}} 
-\fl{\frac {n} {2^{l+2}}},$ 
under these conditions the quotient
\begin{equation*} %\label{eq:quot2}
Q_n=\frac {h_n} 
{2^{e_2(n;l)}}
\end{equation*}
satisfies
\begin{equation} \label{eq:Qrek12} 
Q_n\equiv 2^{-l+2}(s_{2^{l+1}}-s_{2^{l}})Q_{n-2^{l+2}} 
\quad \text{\em mod $2\Z_2$}.
\end{equation}
If 
\hbox{$s_{2^{l}}\not\equiv s_{2^{l+1}}$~{\em mod~$2^{l-1}\Z_2,$}}
then the bound in \eqref{eq:vpH12} is tight for all $n$ which 
are congruent to $0$ or $2^l$ modulo $2^{l+2};$ if
\hbox{$s_{2^{l-1}}\equiv s_{2^{l+1}}+2^{l-2}$~{\em mod~$2^{l}\Z_2,$}}
then also for all $n$ congruent to $2^{l+1}$ modulo~$2^{l+2}$ but not
for those congruent to $3\cdot 2^l,$ while it is the other way round
if
\hbox{$s_{2^{l-1}}\equiv s_{2^{l+1}}+3\cdot2^{l-2}$~{\em mod~$2^{l}\Z_2$}}.
\end{theorem}

\begin{remark}
The reader should observe that, indeed, the conditions of this theorem
do not fit into the framework of Theorem~\ref{thm:pdiv1} with $p=2$.
Namely, the first congruence in \eqref{eq:p^m2a} tells us that we
should choose in addition $m=l-1$ in Theorem~\ref{thm:pdiv1}.
But then, for $i=2^{l+1}$, Condition~\eqref{eq:sdiff}
demands $v_2(\la_{2^{l+1}})\ge -2+l+1=l-1$, while, 
in that case, we have
$$v_2(\la_{2^{l+1}})=
v_2(s_{2^{l+1}}-s_{2^{l-1}})
=v_2\big((s_{2^{l+1}}-s_{2^{l}})+(s_{2^{l}}-s_{2^{l-1}})\big),
$$ 
so that the two congruences in
\eqref{eq:p^m2a} only imply $v_2(\la_{2^{l+1}})$
to be at least $l-2$, which is less by $1$.
\end{remark}

\begin{proof}[Proof of Theorem~{\em \ref{thm:pdiv12}}]
We proceed in the same manner as in the proof of
Theorem~\ref{thm:pdiv1}. Here, we have to bound the $2$-adic valuation
of
\begin{multline*} %\label{eq:H-H32} 
\frac {h_{a2^l+r}} {(a2^l+r)!}=
\sum_{i_0+2^lj_0+2^{l+1}j_1
+\sum_{i=2^l+1}^{2^{l+1}-1} ij_i+\sum_{i=2^{l+1}+1}^\infty ij_i=a2^l+r}
\widetilde H_{i_0}\\
\cdot
\frac {(s_{2^l}-s_{2^{l-1}})^{j_0}} {2^{j_0l}\,j_0!}
\frac {(s_{2^{l+1}}-s_{2^{l-1}})^{j_1}} {2^{j_1(l+1)}\,j_1!}
\underset{i\ne 2^{l+1}}{\prod _{i=2^l+1} ^{\infty}}
\frac {\la_i^{j_i}}
{i^{j_i}\,j_i!}
\end{multline*}
from below. 
The reader should note that the index formerly called $j_{p^{l+1}}$ 
has became $j_1$ here. 
Refining the previous approach, we rearrange the summation
indices according to the sum $j_0+2j_1$:
\begin{multline} \label{eq:H-H32} 
\frac {h_{a2^l+r}} {(a2^l+r)!}=
\sum_{b=0}^a
\sum_{j_0+2j_1=b}
\frac {(s_{2^l}-s_{2^{l-1}})^{j_0}} {2^{j_0l}\,j_0!}
\frac {(s_{2^{l+1}}-s_{2^{l-1}})^{j_1}} {2^{j_1(l+1)}\,j_1!}\\
\cdot
\sum_{i_0+\sum_{i=2^l+1}^{2^{l+1}-1} ij_i
+\sum_{i=2^{l+1}+1}^\infty ij_i=(a-b)2^l+r}
\widetilde H_{i_0}
\underset{i\ne 2^{l+1}}{\prod _{i=2^l+1} ^{\infty}}
\frac {\la_i^{j_i}}
{i^{j_i}\,j_i!}.
\end{multline}
The subsum over $j_0$ and $j_1$, $S_b$ say, equals
\begin{align} 
\notag
S_b&=\sum_{j_0+2j_1=b}
\frac {(s_{2^l}-s_{2^{l-1}})^{j_0}} {2^{j_0l}\,j_0!}
\frac {(s_{2^{l+1}}-s_{2^{l-1}})^{j_1}} {2^{j_1(l+1)}\,j_1!}\\
\notag
&
=\sum_{j_0+2j_1=b}
\frac {\big(2^{-l+1}(s_{2^l}-s_{2^{l-1}})\big)^{j_0}} {2^{j_0}\,j_0!}
\frac {\big(2^{-l+2}(s_{2^{l+1}}-s_{2^{l-1}})\big)^{j_1}} {2^{3j_1}\,j_1!}\\
&
=
2^{-b}\sum_{j_0+2j_1=b}
\frac {\big(2^{-l+1}(s_{2^l}-s_{2^{l-1}})\big)^{j_0}} {j_0!}
\frac {\big(2^{-l+2}(s_{2^{l+1}}-s_{2^{l-1}})\big)^{j_1}} {2^{j_1}\,j_1!}.
\label{eq:j01}
\end{align}
Writing $\mu_0=2^{-l+1}(s_{2^l}-s_{2^{l-1}})$ and
$\mu_2=2^{-l+2}(s_{2^{l+1}}-s_{2^{l-1}})$, we observe that the sum in
the last line equals the coefficient of $z^{b}$ in the series
$f(z)=\exp\left(\mu_0 z+\mu_2\frac {z^2} {2}\right)$. 
Thus, due to \eqref{eq:p^m2a} and \eqref{eq:p^m2b}, we may apply
Corollary~\ref{thm:pdiv1u} with $p=2$, $l=2$, and $m=0$ to $f(z)$.
Consequently,
the $2$-adic valuation of $S_b$ in \eqref{eq:j01} may be bounded
from below by
\begin{equation} \label{eq:pj01}
v_2(S_b)\ge-b
+\left(\fl{\frac {b} {2}}-\fl{\frac {b} {4}}\right)-v_2(b!)=
-b-2\fl{\frac {b} {4}}-\sum_{s\ge3}\fl{\frac {b} {2^s}}, 
\end{equation}
and the quotient
$\widehat Q_b=b!\,S_b/2^{-b+\fl{b/2}-\fl{b/4}}$ 
satisfies
\begin{equation} \label{eq:j01rek}
\widehat Q_b\equiv 2^{-l+2}(s_{2^{l+1}}-s_{2^{l-1}})\widehat Q_{b-4} 
\quad \text{mod $2\Z_2$} .
\end{equation}
Next we turn our attention to the second subsum, say $S_{a,b}$,
\begin{equation*} %\label{eq:Sab} 
S_{a,b}=\sum_{i_0+\sum_{i=2^l+1}^{2^{l+1}-1} ij_i
+\sum_{i=2^{l+1}+1}^\infty ij_i=(a-b)2^l+r}
\widetilde H_{i_0}
\underset{i\ne 2^{l+1}}{\prod _{i=2^l+1} ^{\infty}}
\frac {\la_i^{j_i}}
{i^{j_i}\,j_i!}.
\end{equation*}
By using \eqref{eq:sdiff2}, we see that
the summand in this sum may be bounded from below in the following manner:
{\allowdisplaybreaks
\begin{align}
\notag
-\underset{i\ne 2^{l+1}}
{\sum_{i=2^l+1}^\infty} j_i&\fl{\frac {i} {2^{l}}} 
+\underset{i\ne2^{l+1}}{\sum_{i=2^{l}+1}^{\infty}}
\left(-2^{\fl{\log_2i}-l}+\cl{\frac {i} {2^{l+2}}}+1\right)j_{i}
-\sum_{s\ge1}\underset{i\ne 2^{l+1}}
{\sum_{i=2^l+1}^\infty}\fl{\frac {j_i} {2^s}}
+{\sum_{i=2^{l+1}+1}^\infty} j_i\\
\notag
&\kern1cm
\ge
-\left\lfloor \underset{i\ne 2^{l+1}}
{\sum_{i=2^l+1}^\infty} j_i\frac {i} {2^{l}} \right\rfloor
+\sum_{i=2^{l}+1}^{2^{l+1}-1} j_{i}
-\sum_{i=2^{l+1}+1}^{\infty} 
\left(2^{\fl{\log_2i}-l}-\cl{\frac {i} {2^{l+2}}}-1\right)j_{i}\\
\notag
&\kern3cm
-\sum_{s\ge1}\underset{i\ne 2^{l+1}}
{\sum_{i=2^l+1}^\infty}\fl{\frac {ij_i} {2^{s+l}}}
+\sum_{s\ge1}
{\sum_{i=2^{l+1}+1}^\infty}\left(\fl{\frac {ij_i} {2^{s+l}}}
-
\fl{\frac {j_i} {2^{s}}}\right)\\
\notag
&\kern1cm
\ge
-\left\lfloor\frac {1} {2^l}\underset{i\ne 2^{l+1}}
{\sum_{i=2^l+1}^\infty} ij_i\right\rfloor
+\sum_{i=2^{l}+1}^{2^{l+1}-1} \frac {ij_{i}} {2^{l+2}}
+\sum_{i=2^{l+1}+1}^{\infty} 
\cl{\frac {i} {2^{l+2}}}j_{i}
%\\
%+
%{\sum_{i=2^{l+1}+1}^\infty}\frac {ij_i} {2^{l+2}}
%\notag
%&\kern5cm
-\sum_{s\ge1}\fl{{\underset{i\ne 2^{l+1}}
{\sum_{i=2^l+1}^\infty} \frac {ij_i} {2^{s+l}}}}\\
&\kern1cm
>
-(a-b)
+\cl{\frac {a-b} {4}}
-\sum_{s\ge1}\fl{\frac {a-b} {2^{s}}}
-1,
\label{eq:bound22}
\end{align}}%
where we have used Lemma~\ref{lem:ij} with $p=2$ to obtain the 
next-to-last line,
and the dependence of the summation indices of the sum
in \eqref{eq:H-H32} to obtain the last line.
Moreover, since, in passing from the first to the second line in
\eqref{eq:bound22}, we left out the last term of the first line,
the right-hand side of \eqref{eq:bound22} can be increased by $1$
as soon as one of the $j_i$'s with 
$i\ge 2^{l+1}+1$ is non-zero. Furthermore, when passing
from the first estimate to the second, we used the inequality
$j_i\ge\frac {ij_i} {2^{l+2}}$ for all $i$ with $2^{l}+1\le i\le
2^{l+1}-1$. In particular, for these $i$ we have
$j_i-\frac {ij_i} {2^{l+2}}\ge \frac {j_i} {2}$. Hence, 
the right-hand side of \eqref{eq:bound22} can be increased by $1$
as soon as the sum $\sum_{i=2^l+1}^{2^{l+1}-1}j_i$
should be at least~$2$.

We now use the estimates \eqref{eq:pj01} for $S_b$ and
\eqref{eq:bound22} for $S_{a,b}$ in \eqref{eq:H-H32}, to obtain
{\allowdisplaybreaks
\begin{align}
\notag
v_2(h_{a2^l+r})&\ge 
\sum_{s\ge1}\fl{\frac {a2^l+r} {2^{s}}}
-b-2\fl{\frac {b} {4}}\\
\notag
&\kern3cm
-\sum_{s\ge3}\fl{\frac {b} {2^s}}
-(a-b)
-\sum_{s\ge1}\fl{\frac {a-b} {2^{s}}}
+\cl{\frac {a-b} {4}}\\
\notag
&\ge
\sum_{s=1}^{l-1}\fl{\frac {a2^l+r} {2^{s}}}
+\sum_{s\ge l}\fl{\frac {a2^l+r} {2^{s}}}
-a-2\fl{\frac {b} {4}}\\
\notag
&\kern3cm
-\sum_{s\ge3}\fl{\frac {a} {2^s}}
-\fl{\frac {a-b} {2}}
-\fl{\frac {a-b} {4}}
+\cl{\frac {a-b} {4}}\\
\notag
&\ge
\sum_{s=1}^{l-1}\fl{\frac {a2^l+r} {2^{s}}}
+\sum_{s\ge l+1}\fl{\frac {a2^l} {2^{s}}}
-2\fl{\frac {b} {4}}-\sum_{s\ge3}\fl{\frac {a} {2^s}}
-2\fl{\frac {a-b} {4}}\\
\notag
&\ge
\sum_{s=1}^{l-1}\fl{\frac {a2^l+r} {2^{s}}}
+\sum_{s\ge 1}\fl{\frac {a} {2^{s}}}
-2\fl{\frac {a} {4}}-\sum_{s\ge3}\fl{\frac {a} {2^s}}\\
&\ge
\sum_{s=1}^{l-1}\fl{\frac {a2^l+r} {2^{s}}}
+\fl{\frac {a} {2}}
-\fl{\frac {a} {4}}.
\label{eq:a-b}
\end{align}}%
This is exactly the bound in \eqref{eq:vpH12} with $n=a2^l+r$,
where $0\le r<2^l$.

It should be noted that the remarks after \eqref{eq:bound22} also
show that the bound \eqref{eq:vpH12} for a summand in \eqref{eq:H-H32} 
can only be tight if
$j_i=0$ for $i\ge 2^{l+1}+1$, if at most one of the $j_i$'s with
$2^{l}+1\le i\le 2^{l+1}-1$ is non-zero, and if such a non-zero $j_i$
does not exceed~$1$.
On the other hand, when passing from the second to the third estimate in
\eqref{eq:a-b}, equality occurs only if \hbox{$a-b\not\equiv 1$}~(mod~4).
Let us suppose that we are in the ``adverse" case where $j_{i_0}=1$ for
some $i_0$ between $2^l+1$ and $2^{l+1}-1$ and that all other $j_i$'s
are zero. Then the first line in chain of inequalities
\eqref{eq:bound22} says that the $p$-adic valuation of a summand in
the sum $S_{a,b}$ is at least $0$, while the last line says that
the $p$-adic valuation of such a summand is at least
$$
-(a-b)
+\cl{\frac {a-b} {4}}
-\sum_{s\ge1}\fl{\frac {a-b} {2^{s}}}.
$$
As soon as $a-b\ge2$, the last expression is $\le -1$, and thus
the estimation \eqref{eq:bound22} would not be tight.
On the other hand, we just saw that $a-b=1$ implies that the
estimation \eqref{eq:a-b} would not be tight. The only remaining case
is $a-b=0$, but this contradicts $j_{i_0}=1$ and the dependence of
the summation indices in the sum $S_{a,b}$.
Consequently, for the bound \eqref{eq:vpH12} to be tight for the
corresponding summand (multiplied by $(a2^l+r)!$),
all $j_i$'s with $i\ge 2^{l}+1$ 
must vanish, or, equivalently, $a-b=0$. In other words,
remembering \eqref{eq:pj01}, we have
$$
\frac {h_{a2^l+r}} {(a2^l+r)!}\equiv \widetilde H_rS_a\quad 
\text{mod $2^{-a-2\fl{{a} /{4}}-\sum_{s\ge3}\fl{ {a}/ {2^s}}+1}\Z_2$},
$$
or, recalling that $\widetilde H_r=\frac {h_r} {r!}$,
\begin{align} 
\notag
Q_{a2^l+r}&=h_{a2^l+r}
2^{-\sum_{s=1}^{l-1}\fl{ (a2^l+r)/ {2^s}}-\fl{{a} /{2}}+\fl{ {a}/ {4}}}\\
\notag
&\equiv (a2^l+r)!\,\widetilde H_{r}S_a
2^{-a\sum_{s=1}^{l-1}2^{l-s}
-\sum_{s=1}^{l-1}\fl{ {r}/ {2^s}}-\fl{{a} /{2}}+\fl{ {a}/ {4}}}
\quad \text{mod $2\Z_2$}\\
\notag
&\equiv 
h_{r}\,
2^{-\sum_{s=1}^{l-1}\fl{ {r}/ {2^s}}}
\,a!\,S_a\,
2^{a-\fl{{a} /{2}}+\fl{ {a}/ {4}}}
\quad \text{mod $2\Z_2$}\\
&\equiv Q_r\widehat Q_a\quad \text{mod $2\Z_2$}.
\label{eq:QQ}
\end{align}
Using \eqref{eq:j01rek}, this chain of congruences can be continued,
\begin{align*}
Q_{a2^l+r}&\equiv 2^{-l+2}(s_{2^{l+1}}-s_{2^{l-1}})
Q_r\widehat Q_{a-4}\quad \text{mod $2\Z_2$}\\
&\equiv 2^{-l+2}(s_{2^{l+1}}-s_{2^{l-1}})
Q_{2^l(a-4)+r}\quad \text{mod $2\Z_2$}\\
&\equiv 2^{-l+2}(s_{2^{l+1}}-s_{2^{l}})
Q_{2^la+r-2^{l+2}}\quad \text{mod $2\Z_2$},
\end{align*}
where we have used \eqref{eq:QQ} with $a$ replaced by $a-4$ in the second
line and \eqref{eq:p^m2a} in the last line.
This congruence is equivalent to \eqref{eq:Qrek12} with $n=a2^l+r$.

In order to establish the final assertion, we assume
\hbox{$s_{2^{l}}\not\equiv s_{2^{l+1}}$~{mod~$2^{l-1}\Z_2$}}.
From \eqref{eq:p^m2b}, it then follows that also
\hbox{$s_{2^{l-1}}\not\equiv s_{2^{l}}$~{mod~$2^{l}\Z_2$}}.
The initial condition $Q_0=1$, and the congruence
\begin{align*}
Q_{2^{l}}&\equiv Q_{0}\widehat Q_1\quad \text{mod $2\Z_2$}\\
&\equiv 2^{-l+1}(s_{2^l}-s_{2^{l-1}})\quad \text{mod $2\Z_2$}\\
&\equiv 1\quad \text{mod $2\Z_2$},
\end{align*}
together with \eqref{eq:Qrek12} then imply the assertion for $n$
congruent to $0$ or $2^l$ modulo~$2^{l+2}$. To see the remaining
assertions, let us write $s_{2^{l+1}}=s_{2^{l-1}}+2^{l-2}+\al2^{l-1}$.
Then we have the congruences
\begin{align*}
Q_{2^{l+1}}&\equiv Q_{0}\widehat Q_2\quad \text{mod $2\Z_2$}\\
&\equiv \frac {1} {2}\big(2^{-l+1}(s_{2^l}-s_{2^{l-1}})\big)^2
+\frac {1} {2}2^{-l+2}(s_{2^{l+1}}-s_{2^{l-1}})
\quad \text{mod $2\Z_2$}\\
&\equiv \frac {1} {2}+\frac {1} {2}+\al\quad \text{mod $2\Z_2$}\\
&\equiv 1+\al\quad \text{mod $2\Z_2$},
\end{align*}
and
\begin{align*}
Q_{3\cdot2^{l}}&\equiv Q_{0}\widehat Q_3\quad \text{mod $2\Z_2$}\\
&\equiv \frac {3!} {2}
\bigg(\frac {1} {3!}\big(2^{-l+1}(s_{2^l}-s_{2^{l-1}})\big)^3\\
&\kern2cm
+\frac {1} {2}2^{-l+1}(s_{2^l}-s_{2^{l-1}})
2^{-l+2}(s_{2^{l+1}}-s_{2^{l-1}})\bigg)
\quad \text{mod $2\Z_2$}\\
&\equiv 
2^{-l+1}(s_{2^l}-s_{2^{l-1}})
\left(\frac {1} {2}\big(2^{-l+1}(s_{2^l}-s_{2^{l-1}})\big)^2
+\frac {3} {2}
2^{-l+2}(s_{2^{l+1}}-s_{2^{l-1}})\right)
\quad \text{mod $2\Z_2$}\\
&\equiv \left(\frac {1} {2}
+\frac {3} {2}+3\al
\right)
\quad \text{mod $2\Z_2$}\\
&\equiv \al
\quad \text{mod $2\Z_2$},
\end{align*}
which, together with \eqref{eq:Qrek12}, complete the proof of the theorem.
\end{proof}

\section{$p$-Divisibility of coefficients in exponentials of
  power series, II} \label{sec:pdiv2}

\noindent In this section we present our second set of 
main results. They address the case where we have less precise
information on the coefficients $s_n$ of the series $S(z)$
in \eqref{eq:H-S}. More specifically, what can we say if 
Condition~\eqref{eq:pS1} is ``truncated" to hold only up to the
coefficient of $z^{p^{l}}$, but beyond this threshold
we do not assume any additional
information, except that all $s_n$'s should lie in $\Z_p$?
While we cannot expect a result as strong as the bound \eqref{eq:vpH10}
in Corollary~\ref{thm:pdiv10} with $l$ replaced by $l+1$, 
the next theorem tells us that the
$p$-divisibility of the coefficients of the exponential $H(z)$ in
\eqref{eq:H-S} is still higher than one might expect.

\begin{theorem} \label{thm:pdiv2}
For a prime number $p\ge3,$
let\/ $S(z)=\sum_{n\ge1}\frac {s_n} {n}z^n$ be a formal power series
with $s_n\in\Z_p$ for all $n,$ and let 
$H(z)=\sum_{n\ge0}\frac {h_n} {n!}z^n$ be the exponential of $S(z)$.
Given a positive integer $l$
such that $(p,l)\ne(3,1),$ we assume that
\begin{equation} \label{eq:pS2}
S(z^{p})-pS(z)=pJ(z)
+O\left(z^{p^{l}+1}\right) 
\end{equation}
with $J(z)\in \Z_p[z]$.
Then
\begin{equation} \label{eq:vpH2}
v_p(h_n)\ge \sum_{s \ge 1}\fl{\frac {n} {p^{s}}} 
-(l-1)\fl{\frac {n} {p^l}}-\sum_{s \ge l}\fl{\frac {n} {2p^{s}}} 
\end{equation}
for all $n$.
\end{theorem}

\begin{remarks}
(1)
We should compare the bound in \eqref{eq:vpH2} to the one in \eqref{eq:vpH10},
in two different ways.

First, let us consider Corollary~\ref{thm:pdiv10}, and compare
it to Theorem~\ref{thm:pdiv2}. The difference in conditions
is that, while, under the assumptions of Corollary~\ref{thm:pdiv10}, 
Condition~\eqref{eq:pS1d} is (potentially) 
violated starting from the coefficient of
$z^{p^l}$ on, under the assumptions of Theorem~\ref{thm:pdiv2}
it is (potentially) violated starting from the coefficient of 
$z^{p^l+1}$ ``only."
The ``gain" in $p$-divisibility is given by the difference between
\eqref{eq:vpH2} and \eqref{eq:vpH10}, that is, by
$$
\sum_{s \ge l}\left(\fl{\frac {n} {p^{s}}}-\fl{\frac {n} {2p^{s}}}\right).
$$

On the other hand, let us consider Corollary~\ref{thm:pdiv10} with $l$
replaced by $l+1$, and compare
this to Theorem~\ref{thm:pdiv2}. Now the difference in conditions
is that, under the assumptions of Corollary~\ref{thm:pdiv10}, 
Condition~\eqref{eq:pS1d} is (potentially) violated  
starting from the coefficient of
$z^{p^{l+1}}$ on, while, under the assumptions of Theorem~\ref{thm:pdiv2},
it is still (potentially) violated already 
starting from the coefficient of $z^{p^l+1}$ on.
The ``loss" in $p$-divisibility is then given by the difference between
\eqref{eq:vpH10} with $l$ replaced by $l+1$ and \eqref{eq:vpH2}, that is, by
$$
(l-1)\fl{\frac {n} {p^l}}+\fl{\frac {n} {2p^{l}}}
-(l+1)\fl{\frac {n} {p^{l+1}}}+\fl{\frac {n} {2p^{l+1}}}
+\sum_{s \ge l+2}\left(\fl{\frac {n} {2p^{s}}}
-\fl{\frac {n} {p^{s}}} \right).
$$

\medskip
(2) The bound in \eqref{eq:vpH2} could be further improved,
albeit very likely at the cost of having to introduce very
complicated arithmetic functions. This is the reason why we have 
refrained from trying to formulate an improved version of
Theorem~\ref{thm:pdiv2}. 
That the bound \eqref{eq:vpH2}
cannot be tight for infinitely many~$n$, can be seen from the
proof: when going from the next-to-last to the last estimate
in \eqref{eq:Hest}, the inequality
$$
-\sum_{i=p^l+1}^\infty j_i(l-1)
\ge-a(l-1)
$$
is used. This inequality will be ``very far from equality" if
$n=ap^l+r$ is large.
\end{remarks}

\begin{proof}[Proof of Theorem~{\em \ref{thm:pdiv2}}]
Following again the proof of Theorem~\ref{thm:pdiv1},
we let $n=ap^l+r$ with $0\le r<p^l$.
Here, instead of \eqref{eq:H-H3}, we obtain
\begin{equation} \label{eq:H-H4} 
\frac {h_{ap^l+r}} {(ap^l+r)!}=
\sum_{i_0+\sum_{i=p^l+1}^\infty ij_i=ap^l+r}
\widetilde H_{i_0}
\prod _{i=p^l+1} ^{\infty}
\frac {\la_i^{j_i}}
{i^{j_i}\,j_i!},
\end{equation}
with $\widetilde H_{i_0}\in \Z_p$ for all $i_0$, and the $\la_i$'s
being defined as in the statement of Theorem~\ref{thm:pdiv1}. Since
here we are assuming that all $s_i$'s are in $\Z_p$ (instead of
$\Q_p$, as in Theorem~\ref{thm:pdiv1}), we have $\la_i\in \Z_p$
for all $i$. For the $p$-adic valuation of the summand in the above
sum, we then obtain
{\allowdisplaybreaks
\begin{align}
\notag
v_p&\left(\widetilde H_{i_0}
\prod _{i=p^l+1} ^{\infty}
\frac {\la_i^{j_i}}
{i^{j_i}\,j_i!}\right)
\ge -\sum_{i=p^l+1}^\infty j_i\cdot v_p(i)
-\sum_{i=p^l+1}^\infty\sum_{s\ge1}\fl{\frac {j_i} {p^s}}\\
\notag
&\kern1cm
\ge
-\sum_{i=p^l+1}^{2p^l-1} j_i(l-1)
-\sum_{i=2p^l}^\infty j_i\cdot v_p(i)
-\sum_{i=p^l+1}^{2p^l-1}\sum_{s\ge1}\fl{\frac {j_i} {p^s}}
-\sum_{i=2p^l}^\infty\sum_{s\ge1}\fl{\frac {j_i} {p^s}}\\
\notag
&\kern1cm
\ge
-\sum_{i=p^l+1}^{2p^l-1} j_i(l-1)
-\sum_{i=2p^l}^\infty j_i\left(l-1+\frac {i} {2p^l}\right)
-\chi(p=3)\frac {j_{p^{l+1}}}2 \\
\notag
&\kern5cm
-\sum_{i=p^l+1}^{2p^l-1}\frac {j_i} {p-1}
-\sum_{s\ge1}\sum_{i=2p^l}^\infty\fl{\frac {ij_i} {2p^{s+l}}}\\
%+\fl{\frac {j_{p^{l+1}}} {2}}-\fl{\frac {j_{p^{l+1}}} {p}}\\
\notag
&\kern1cm
>
-\sum_{i=p^l+1}^\infty j_i(l-1)
-\Bigg\lfloor\sum_{i=p^l}^\infty j_i\frac {i} {2p^l}\Bigg\rfloor-1\\
\notag
&\kern3cm
+\sum_{i=p^l+1}^{2p^l-1}\left(\frac {ij_i} {2p^l}
-\frac {j_i} {p-1}\right)
-\sum_{s\ge1}\Bigg\lfloor\sum_{i=2p^l}^\infty\frac {ij_i}
{2p^{s+l}}\Bigg\rfloor
-\chi(p=3)\frac {j_{p^{l+1}}}2\\
&\kern1cm
>
-a(l-1)+(l-1)j_{p^{l+1}}
-\fl{\frac {a} {2}}
-\sum_{s\ge1}\fl{\frac {a} {2p^{s}}}-\chi(p=3)\frac {j_{p^{l+1}}}2
-1.
\label{eq:Hest}
\end{align}}%
Here, we have again used Legendre's formula \cite[p.~10]{LegeAA} in
the first step, the estimation
$$
\sum_{s\ge1}\fl{\frac {j} {p^s}}\le 
\sum_{s\ge1}\frac {j} {p^s}\le \frac {j} {p-1}
$$
in the third step, and the dependence of the summation indices of the sum
in \eqref{eq:H-H4} in the last step. The strict inequality in the 
fourth step results from the inequality $-\al>-\fl\al-1$. 
By assumption, we have $(p,l)\ne (3,1)$, and therefore
$$
(l-1)j_{p^{l+1}}-\chi(p=3)\frac {j_{p^{l+1}}}2\ge0.$$
Thus, out of \eqref{eq:Hest} we obtain exactly the bound in
\eqref{eq:vpH2} with $n=ap^l+r$, as desired.
\end{proof}

By the same approach, we also obtain a corresponding result for
the exceptional case where $p=3$ and $l=1$.

\begin{theorem} \label{thm:pdiv31}
Let\/ $S(z)=\sum_{n\ge1}\frac {s_n} {n}z^n$ be a formal power series
with $s_n\in\Z_3$ for all $n,$ and let 
$H(z)=\sum_{n\ge0}\frac {h_n} {n!}z^n$ be the exponential of $S(z)$.
We assume that
\begin{equation} \label{eq:pS31}
S(z^{3})-3S(z)=3J(z)
+O\left(z^{4}\right) 
\end{equation}
with $J(z)\in \Z_3[z]$.
Then
\begin{equation} \label{eq:vpH31}
v_3(h_n)\ge \sum_{s \ge 1}
\left(\fl{\frac {n} {3^{s}}} 
-\fl{\frac {n} {2\cdot3^{s}}} \right)
-\fl{\frac {n} {18}} 
\end{equation}
for all $n$.
\end{theorem}

By refining the above approach, and using ideas from the proof of
Theorem~\ref{thm:pdiv12}, we also obtain a corresponding result
for $p=2$, a case which was excluded in both Theorems~\ref{thm:pdiv2}
and~\ref{thm:pdiv31}. Again, there is a technical auxiliary result
which is needed for the proof. It is given separately in
Lemma~\ref{lem:32j}.

\begin{theorem} \label{thm:pdiv22}
Let\/ $S(z)=\sum_{n\ge1}\frac {s_n} {n}z^n$ be a formal power series
with $s_n\in\Z_2$ for all $n,$ and let 
$H(z)=\sum_{n\ge0}\frac {h_n} {n!}z^n$ be the exponential of $S(z)$.
Given a positive integer $l,$ we assume that
\begin{equation} \label{eq:pS22}
S(z^{2})-2S(z)=2J(z)
+O\left(z^{2^{l}+1}\right) 
\end{equation}
with $J(z)\in \Z_2[z]$.
Then
\begin{equation} \label{eq:vpH22}
v_2(h_n)\ge
\begin{cases} 
\fl{\frac {n} {2}}
-\fl{\frac {n} {4}},
&\text{if\/ }l=1,\\
\fl{\frac {n} {2}},
&\text{if\/ }l=2,\\
\sum_{s=1} ^{l+1}\fl{\frac {n} {2^{s}}}
-(l-1)\fl{\frac {n} {2^l}},
&\text{if\/ }l\ge3,
\end{cases}
\end{equation}
for all $n$.
\end{theorem}

\begin{remarks}
(1) The bounds in \eqref{eq:vpH22} in the cases where $l=1$ or $l=2$
are worse than the bound in \eqref{eq:vpH22} in the generic case
(third line on the right-hand side), as can be seen by inspection.
The fact that, in the case where $l=1$, the bound cannot be improved
is provided for by the observation that, in that case,
Corollary~\ref{thm:pdiv10} with $p=2$ and
$l=2$ applies, and hence the tightness assertion given there holds.
On the other hand, 
the fact that, in the case where $l=2$, the bound cannot be improved
can be seen by considering e.g.\ $S(z)=z+\frac {z^2} {2}+\frac {z^4} {4}$.
Namely, in that case we have $v_2(h_{8m})=4m$, since in
the proof below (see \eqref{eq:Hest2b}) all terms except the
one with $j_8=m$ (and all other $j_i$'s equal to zero) have
a higher $2$-divisibility.

\medskip
(2) For $l\ge3$, the bound in \eqref{eq:vpH22} is actually better
than the bound in \eqref{eq:vpH2} with $p=2$. Namely, the difference
between the former and the latter is
$\fl{{n} /{2^{l+1}}}$. For $l=2$, the bound in \eqref{eq:vpH22} is
just barely better than the bound in \eqref{eq:vpH2}, the difference
being $\fl{n/4}-2\fl{n/8}$, which equals $0$ or $1$, depending on
whether $n\equiv0,1,2,3$~(mod~$8$) or not. 
For $l=1$, the bound in \eqref{eq:vpH22} is worse
than the bound in \eqref{eq:vpH2} with $p=2$. 
\end{remarks}

\begin{proof}[Proof of Theorem~{\em \ref{thm:pdiv22}}]
If $l=1$, then we apply Corollary~\ref{thm:pdiv10} with $p=2$ and
$l=2$. Since, for this choice of $p$ and $l$, Condition~\eqref{eq:pS10} is
equivalent to \eqref{eq:pS22} with $l=1$, we immediately get 
the corresponding bound in \eqref{eq:vpH22} from \eqref{eq:vpH10}.

We postpone the discussion of the case where $l=2$, and, for the moment,
focus 
on the generic case where $l\ge3$. Proceeding in the same manner
as in the proof of Theorem~\ref{thm:pdiv2},
we let $n=a2^l+r$ with $0\le r<2^l$. We then must bound
the $2$-adic valuation of the summands in the series
\begin{equation} \label{eq:H-H5} 
\frac {h_{a2^l+r}} {(a2^l+r)!}=
\sum_{i_0+\sum_{i=2^l+1}^\infty ij_i=a2^l+r}
\widetilde H_{i_0}
\prod _{i=2^l+1} ^{\infty}
\frac {\la_i^{j_i}}
{i^{j_i}\,j_i!}
\end{equation}
from below,
where we know that $\widetilde H_{i_0}\in \Z_2$ for all $i_0$, 
and where the $\la_i$'s
are defined as in the statement of Theorem~\ref{thm:pdiv1} with
$p=2$. Again, since
we are assuming that all $s_i$'s are in $\Z_2$ (instead of
$\Q_2$ as in Theorem~\ref{thm:pdiv1}), we have $\la_i\in \Z_2$
for all $i$. For the $2$-adic valuation of the summand in the above
sum, we then find that
{\allowdisplaybreaks
\begin{align}
\notag
v_2&\left(\widetilde H_{i_0}
\prod _{i=2^l+1} ^{\infty}
\frac {\la_i^{j_i}}
{i^{j_i}\,j_i!}\right)
\ge -\sum_{i=2^l+1}^\infty j_i\cdot v_2(i)
-\sum_{i=2^l+1}^\infty\sum_{s\ge1}\fl{\frac {j_i} {2^s}}\\
\notag
&\kern1cm
\ge
-\sum_{i=2^l+1}^{2^{l+1}-1} j_i(l-2)-j_{3\cdot 2^{l-1}}
-\sum_{i=2^{l+1}}^\infty j_i(l-1)
-\sum_{i=2^{l+1}}^\infty j_i(v_2(i)-l+1)\\
\notag
&\kern5cm
-\underset{i\ne 3\cdot 2^{l-1}}
{\sum_{i=2^l+1}^{2^{l+1}-1}}\sum_{s\ge1}\fl{\frac {j_i} {2^s}}
-\sum_{s\ge1}\fl{\frac {j_{3\cdot 2^{l-1}}} {2^s}}
-\sum_{i=2^{l+1}}^\infty\sum_{s\ge1}\fl{\frac {j_i} {2^s}}\\
\notag
&\kern1cm
\ge
-\sum_{i=2^l+1}^\infty j_i(l-1)
+\sum_{i=2^l+1}^{2^{l+1}-1} j_i
-j_{3\cdot 2^{l-1}}
-\sum_{i=2^{l+1}}^\infty j_i(v_2(i)-l+1)
\\
\notag
&\kern3cm
-\underset{i\ne 3\cdot 2^{l-1}}
{\sum_{i=2^l+1}^{2^{l+1}-1}}{j_i} 
-\fl{\frac {j_{3\cdot 2^{l-1}}} {2}}
-\sum_{s\ge1}\fl{\frac {3\cdot 2^{l-1}j_{3\cdot 2^{l-1}}} {2^{s+l+1}}}
-\sum_{s\ge1}\sum_{i=2^{l+1}}^\infty\fl{\frac {ij_i} {2^{s+l+1}}}\\
\notag
&\kern1cm
\ge
-(l-1)\Bigg(j_{3\cdot 2^{l-1}}
+{\underset{i\ne 3\cdot 2^{l-1}}
{\sum_{i=2^l+1}^\infty} \fl{\frac {i} {2^l}}j_i
}\Bigg)
+
{\sum_{i=2^{l+1}}^\infty} j_i(l-1)\left(\fl{\frac {i} {2^l}}-1\right)
\\
\notag
&\kern5cm
-\sum_{i=2^{l+1}}^\infty j_i(v_2(i)-l+1)
-\fl{\frac {j_{3\cdot 2^{l-1}}} {2}}
\\
\notag
&\kern5cm
-\sum_{s\ge1}\Bigg\lfloor
{\frac {3\cdot 2^{l-1}j_{3\cdot 2^{l-1}}} {2^{s+l+1}}}
+\sum_{i=2^{l+1}}^\infty\frac {ij_i}
{2^{s+l+1}}\Bigg\rfloor\\
\notag
&\kern1cm
\ge
-(l-1)\left(j_{3\cdot 2^{l-1}}
+\fl{\underset{i\ne 3\cdot 2^{l-1}}
{\sum_{i=2^l+1}^\infty} {\frac {i} {2^l}}j_i
}\right)
-\sum_{s\ge1}\fl{\frac {a} {2^{s+1}}}\\
&\kern3cm
-\fl{\frac {j_{3\cdot 2^{l-1}}} {2}}
%+\frac {l-1} {2}\fl{\frac {j_{3\cdot 2^{l-1}}} {2}}
+\sum_{i=2^{l+1}}^\infty j_i\left(\fl{\frac {i} {2^l}}(l-1)-v_2(i)\right)
,
\label{eq:Hest2}
\end{align}}%
where we have used
the dependence of the summation indices of the sum
in \eqref{eq:H-H5} in the last line. It is easy to see that
$$
2^{x-l}(l-1)\ge x
$$
for all real $x$ with $x\ge l+1$ and $l\ge3$. 
Hence, with $x=\fl{\log_2 i}$, we have
$$
\fl{\frac {i} {2^l}}(l-1)-v_2(i)\ge
2^{x-l}(l-1)-x\ge0,
$$
so that the last term in the last line of \eqref{eq:Hest2} is 
non-negative. Next, we apply Lemma~\ref{lem:32j} with
$j=j_{3\cdot 2^{l-1}}$ and $x$ equal to the sum 
${\sum_{i=2^l+1,\,i\ne 3\cdot 2^{l-1}}^\infty} {{i} }j_i/2^l$.
This leads to the estimation
\begin{align} 
\notag
v_2\left(\widetilde H_{i_0}
\prod _{i=2^l+1} ^{\infty}
\frac {\la_i^{j_i}}
{i^{j_i}\,j_i!}\right)
&\ge
-(l-1)\fl{
\sum_{i=2^l+1}^\infty \frac {i} {2^l}j_i
}
-\sum_{s\ge1}\fl{\frac {a} {2^{s+1}}}
+\frac {l-3} {2}\fl{\frac {j_{3\cdot 2^{l-1}}} {2}}
\\
&\ge
-(l-1)a
-\sum_{s\ge1}\fl{\frac {a} {2^{s+1}}}
,
\label{eq:Hest2a}
\end{align}
which is exactly the bound \eqref{eq:vpH22} with $l\ge3$ and
$n=a2^l+r$, where $0\le r<2^l$.

For the remaining case where $l=2$, a weakened version of the
estimation \eqref{eq:Hest2} suffices. Here, we set $n=4a+r$, with
$0\le r<4$. Then
\begin{align}
\notag
v_2&\left(\widetilde H_{i_0}
\prod _{i=5} ^{\infty}
\frac {\la_i^{j_i}}
{i^{j_i}\,j_i!}\right)
\ge -\sum_{i=5}^\infty j_i\cdot v_2(i)
-\sum_{i=5}^\infty\sum_{s\ge1}\fl{\frac {j_i} {2^s}}\\
\notag
&\kern1cm
\ge
-j_{6}
-\sum_{i=8}^\infty j_i\cdot v_2(i)
-\sum_{i=5}^7\sum_{s\ge1}\fl{\frac {j_i} {2^s}}
-\sum_{i=8}^\infty\sum_{s\ge1}\fl{\frac {j_i} {2^s}}\\
\notag
&\kern1cm
\ge
-\sum_{i=5}^7 j_i
-\sum_{i=8}^\infty j_i\left(\fl{\frac {i} {4}}+\fl{\frac {i} {8}}\right)
+\sum_{i=8}^\infty j_i\left(3\fl{\frac {i} {8}}-v_2(i)\right)
\\
\notag
&\kern4cm
-\sum_{i=5}^7\fl{\frac {j_i} {2}}
-\sum_{i=5}^7\sum_{s\ge1}\fl{\frac {ij_{i}} {2^{s+3}}}
-\sum_{s\ge1}\sum_{i=8}^\infty\fl{\frac {ij_i} {2^{s+3}}}\\
\notag
&\kern1cm
\ge
-\Bigg\lfloor
\sum_{i=5}^\infty\frac {ij_i}
{4}\Bigg\rfloor
-\Bigg\lfloor
\sum_{i=5}^\infty\frac {ij_i}
{8}\Bigg\rfloor
+{\sum_{i=8}^\infty} j_i\left(3\fl{\frac {i} {8}}-v_2(i)\right)
-\sum_{s\ge1}\Bigg\lfloor
\sum_{i=5}^\infty\frac {ij_i}
{2^{s+3}}\Bigg\rfloor\\
&\kern1cm
\ge
-\fl{\frac {n} {4}}
-\fl{\frac {n} {8}}
-\sum_{s\ge1}\fl{\frac {n} {2^{s+3}}}
+{\sum_{i=8}^\infty} j_i\left(3\fl{\frac {i} {8}}-v_2(i)\right)
,
\label{eq:Hest2b}
\end{align}
where we have used
the dependence of the summation indices of the sum
in \eqref{eq:H-H5} with $l=2$ in the last line. 
To finish up, one applies the inequality
$$
3\fl{\frac {i} {8}}-v_2(i)\ge
3\cdot 2^{x-3}-x\ge0,\quad \text{for }i\ge8,
$$
where $x=\fl{\log_2 i}$. If this is used in \eqref{eq:Hest2b},
then one obtains the bound in \eqref{eq:vpH22}
with $l=2$.
This completes the proof of the theorem.
\end{proof}

\begin{lemma} \label{lem:32j}
For all integers $j$ and real numbers $x,$ we have
$$
j+\fl x\le 
\fl{\frac {3} {2}j+x}-\frac {1} {2}\fl{\frac {j} {2}}.
$$
\end{lemma}

\begin{proof}If $j$ is even, say $j=2J$, then 
$$
\fl{\frac {3} {2}j+x}-\frac {1} {2}\fl{\frac {j} {2}}
=\fl{3J+x}-\frac {1} {2}J\ge 2J+\fl x=j+\fl x.
$$
On the
other hand, if $j$ is odd, say $j=2J+1$, then
$$
\fl{\frac {3} {2}j+x}-\frac {1} {2}\fl{\frac {j} {2}}
=
\fl{3J+\frac {3} {2}+x}-\frac {1} {2}J
\ge 2J+1+\fl{x}=j+\fl x,
$$
as desired.
\end{proof}

As announced in Remark~(4) after Theorem~\ref{thm:pdiv1},
there is a sharp dividing line concerning $p$-divisibility
of the coefficients of the series $H(z)$ depending on whether
$s_{1}\equiv s_{p}$~mod~$p\Z_p$ or not.

\begin{corollary} \label{cor:divline}
For a prime number $p,$
let\/ $S(z)=\sum_{n\ge1}\frac {s_n} {n}z^n$ be a formal power series
with $s_n\in\Q_p$ for all $n,$ and let 
$H(z)=\sum_{n\ge0}\frac {h_n} {n!}z^n$ be the exponential of $S(z)$.
If $s_{1}\equiv s_{p}$~{\em mod}~$p\Z_p,$ then, for $p\ge5,$ we have
\begin{equation} \label{eq:divl1}
v_p(h_n)\ge 
\sum_{s \ge 1}\left(\fl{\frac {n} {p^{s}}}
-\fl{\frac {n} {2p^{s}}}\right),
\end{equation}
for $p=3,$ we have
\begin{equation} \label{eq:divl3}
v_3(h_n)\ge 
\sum_{s \ge 1}\left(\fl{\frac {n} {3^{s}}}
-\fl{\frac {n} {2\cdot3^{s}}}\right)-\fl{\frac {n} {18}},
\end{equation}
while, for $p=2,$ we have
\begin{equation} \label{eq:divl2}
v_2(h_n)\ge \fl{\frac {n} {2}}-\fl{\frac {n} {4}}.
\end{equation}
On the other hand, 
if \hbox{$s_{1}\not\equiv s_{p}$~{\em mod}~$p\Z_p,$} then for each $N$
there exists some $n>N$ such that $h_n$ is not divisible by $p$.
\end{corollary}

If we assume that Condition~\eqref{eq:pS1d} is even satisfied up to
the coefficient of $z^{p^{2p^l-1}}$, then Theorem~\ref{thm:pdiv2} may
be further improved.

\begin{theorem} \label{thm:pdiv3}
For a prime number $p\ge3,$
let\/ $S(z)=\sum_{n\ge1}\frac {s_n} {n}z^n$ be a formal power series
with $s_n\in\Z_p$ for all $n,$ and let 
$H(z)=\sum_{n\ge0}\frac {h_n} {n!}z^n$ be the exponential of $S(z)$.
Given a positive integer $l$
such that $(p,l)\ne(3,1),$ we assume that
\begin{equation} \label{eq:pS3}
S(z^{p})-pS(z)=pJ(z)
+O\left(z^{2p^{l}}\right) 
\end{equation}
with $J(z)\in \Z_p[z]$.
Then
\begin{equation} \label{eq:vpH3}
v_p(h_n)\ge \sum_{s \ge 1}\fl{\frac {n} {p^{s}}}
-(l-1)\cl{\frac {n} {2p^l}}
-\sum_{s \ge l}\fl{\frac {n} {2p^{s}}} 
\end{equation}
for all $n$.
\end{theorem}

\begin{remarks}
(1) The ``gain" effected by the stronger conditions in
  Theorem~\ref{thm:pdiv3} over those in Theorem~\ref{thm:pdiv2} is measured
by the difference between the bounds in \eqref{eq:vpH3} and
\eqref{eq:vpH2}, namely
$$
(l-1)\left(\fl{\frac {n} {p^l}}-\cl{\frac {n} {2p^l}}\right).
$$

\medskip
(2) There is no need to have a $p=2$ version of
  Theorem~\ref{thm:pdiv3} since this is given by
  Corollary~\ref{thm:pdiv10} with $p=2$ and $l$ replaced by $l+1$.

\medskip
(3) If one would work through the proof below, with $p=3$ and
$l=1$, then one would not obtain any improvement over Theorem~\ref{thm:pdiv31}.
\end{remarks}

\begin{proof}[Proof of Theorem~{\em \ref{thm:pdiv3}}]
The proof runs along the lines of the proof of Theorem~\ref{thm:pdiv2}.
The relevant computation is:
\begin{align}
\notag
v_p&\left(\widetilde H_{i_0}
\prod _{i=2p^l} ^{\infty}
\frac {\la_i^{j_i}}
{i^{j_i}\,j_i!}\right)
\ge -\sum_{i=2p^l}^\infty j_i\cdot v_p(i)
-\sum_{i=2p^l}^\infty\sum_{s\ge1}\fl{\frac {j_i} {p^s}}\\
\notag
&\kern1cm
\ge
-\sum_{i=2p^l}^\infty j_i\left(l-1+\frac {i} {2p^l}\right)
-\chi(p=3)\frac {j_{p^{l+1}}}2 
-\sum_{s\ge1}\sum_{i=2p^l}^\infty\fl{\frac {ij_i} {2p^{s+l}}}\\
%+\fl{\frac {j_{p^{l+1}}} {2}}-\fl{\frac {j_{p^{l+1}}} {p}}\\
\notag
&\kern1cm
>
-\sum_{i=2p^l}^\infty j_i\frac {i} {2p^l}(l-1)
+(l-1)\left(\frac {p} {2}-1\right)j_{p^{l+1}}
-\Bigg\lfloor\sum_{i=2p^l}^\infty j_i\frac {i} {2p^l}\Bigg\rfloor-1\\
\notag
&\kern5cm
-\sum_{s\ge1}\Bigg\lfloor\sum_{i=2p^l}^\infty\frac {ij_i}
{2p^{s+l}}\Bigg\rfloor
-\chi(p=3)\frac {j_{p^{l+1}}}2\\
\notag
&\kern1cm
>
-(l-1)\cl{\frac {a} {2}}
-\fl{\frac {a} {2}}
-\sum_{s\ge1}\fl{\frac {a} {2p^{s}}}\\
&\kern5cm
+(l-1)\frac {j_{p^{l+1}}} {2}
-\chi(p=3)\frac {j_{p^{l+1}}}2
-1.
\label{eq:Hest3}
\end{align}
Again, it is seen that this implies the bound in \eqref{eq:vpH3}.
\end{proof}

\section{Counting subgroups in finitely generated groups}
\label{sec:fingen}

\noindent Let $p$ be a prime number.
The main result of this section, Proposition~\ref{Prop:FinGenIndexp},
in particular shows that the number $s_p(\Gamma)$ of subgroups of
index~$p$ in a finitely generated group $\Gamma$ is always congruent
to $0$ or $1$ modulo~$p$. As we are going to see in the next section, 
in combination with Corollary~\ref{cor:divline}, 
this leads to a sharp dividing line between
the two cases concerning the $p$-divisibility of the homomorphism
numbers $h_n(\Gamma)$; see Theorem~\ref{thm:fingen}. As illustration of
Proposition~\ref{Prop:FinGenIndexp}, we exhibit two classes of
finite groups $G$ for which $s_p(G)\equiv0$~(mod~$p$) (and, more
generally, even $s_{p^m}(G)\equiv0$~(mod~$p$) for each prime power
$p^m$ dividing the order of~$G$) in Corollaries~\ref{cor:pm1}
and \ref{cor:pm2}.

\medskip
Let $\Gamma$ be a finitely generated group, let $p$
be a prime number, and let $m$ be a positive integer. We write the
primary decomposition of the finitely generated Abelian group
$\bar{\Gamma} = \Gamma/[\Gamma, \Gamma]$ as 
\[
\bar{\Gamma} \cong \bigoplus_{q\in\mathbb{P}}\,
\bigoplus_{\mu=1}^{M_q} C_{q^{\mu}}^{e^{(q)}_{\mu}(\Gamma)}\,
\oplus\, C_\infty^{\bar{r}_\infty(\Gamma)}, 
\]
where $\mathbb{P}$ denotes the set of positive rational primes. For
$q\in\mathbb{P}$, the number 
\[
\bar{r}_q(\Gamma) := e_1^{(q)}(\Gamma) + e_2^{(q)}(\Gamma) + \cdots +
e^{(q)}_{M_q}(\Gamma), 
\]
i.e., the rank (minimal number of generators) of the $q$-part of
$\bar{\Gamma}$, is called the \emph{$q$-rank} of the finitely
generated Abelian group $\bar{\Gamma}$. 

For a positive integer $k$, denote by $n_{k}(\Gamma)$ the number of
normal subgroups of index $k$ in $\Gamma$. The following observation
will be useful. 

\begin{lemma}
\label{Lem:NormalCount}
For a prime number $p,$ we have
\begin{equation}
\label{Eq:NormalpCount}
n_p(\Gamma) = \frac{p^{\bar{r}_p(\Gamma) + \bar{r}_\infty(\Gamma)}-1}{p-1}.
\end{equation}
In particular, $n_p(\Gamma)$ either equals zero, or is congruent to
$1$ modulo~$p,$ depending on whether or not $\bar{r}_p(\Gamma) +
\bar{r}_\infty(\Gamma)$ vanishes. 
%Also,  
%\begin{equation}
%\label{Eq:Normalp2Count}
%n_{p^2}(\Gamma) = \frac{p^{\bar{r}_p(\Gamma) +
%    \bar{r}_\infty(\Gamma)}\left[(p^2-1) p^{\bar{r}_p(\Gamma) +
%      \bar{r}_\infty(\Gamma)-e^{(p)}_1} + p^{\bar{r}_p(\Gamma) +
%      \bar{r}_\infty(\Gamma)} - p(p+1)\right] + p}{p(p-1)^2(p+1)}. 
%\end{equation}
\end{lemma}

\begin{proof}
This is \cite[Lemma~2]{MueFor}.
\end{proof}

With Lemma~\ref{Lem:NormalCount} at our disposal, we can now show the
following. 

\begin{proposition}
\label{Prop:FinGenIndexp}
%\begin{enumerate}
%\item[(i)] 
Let\/ $\Gamma$ be a finitely generated group, and let $p$ be a prime number.
The number $s_p(\Gamma)$ of index $p$ subgroups in $\Gamma$ 
does not attain the values $2, 3,
  \ldots, p-1$ modulo $p$. More precisely, if\/ $\Gamma$ contains a
  subgroup of index $p,$ we have 
\begin{equation}
\label{Eq:FinGenIndpCong}
s_p(\Gamma) \equiv \begin{cases} 0~(\mathrm{mod}~p),& 
\text{if\/ }\bar{r}_p(\Gamma) +
  \bar{r}_\infty(\Gamma) = 0,\\[1mm] 
1~(\mathrm{mod}~p),& 
\text{if\/ }\bar{r}_p(\Gamma) + \bar{r}_\infty(\Gamma)
>0.\end{cases}
\end{equation}
Also, if\/ $\Gamma$ does not contain a normal subgroup of
  index $p^m$ for some $m\geq1,$ then 
$s_{p^m}(\Gamma) \equiv 0~(\mathrm{mod}~p)$. 
%\end{enumerate}
\end{proposition}

\begin{proof}
%(i)
Let $\Gamma$ be a finitely generated group, and let $p^m$ be a
  non-trivial prime power. Let $\mathcal{C}_{p^m}(\Gamma)$ be the
  complex of subgroups of index $p^m$ in $\Gamma$. Since $\Gamma$ is
  assumed to be finitely generated, $\mathcal{C}_{p^m}(\Gamma)$ is
  finite. The group $\Gamma$ acts from the right
  on this finite complex  by conjugation, i.e., via 
\[
\Delta\cdot \gamma := \gamma^{-1} \Delta
\gamma,\qquad\gamma\in\Gamma,\, \Delta\in \mathcal{C}_{p^m}(\Gamma). 
\] 
Consider a group
$\Delta\in\mathcal{C}_{p^m}(\Gamma)$. The orbit of $\Delta$ under this
$\Gamma$-action (i.e., its conjugacy class) has size $(\Gamma:
N_\Gamma(\Delta))$. Since $\Delta\subseteq N_\Gamma(\Delta)$,
this size must be a power of $p$, say $p^e$, with $0\le e\le m$. 
If $e=0$, then $N_\Gamma(\Delta))=\Gamma$,
that is, $\Delta$ is normal and the size of its orbit is $1$. 
Otherwise the orbit size is divisible
by $p$. Thus,
\begin{equation} 
\label{Eq:s/n}
s_{p^m}(\Gamma) \equiv n_{p^m}(\Gamma)~(\mathrm{mod}~p).
\end{equation}
The claims now follow by combining \eqref{Eq:s/n} with 
%the first part of
Lemma~\ref{Lem:NormalCount}. 
%
%
%(ii) This follows in a similar manner by combining \eqref{Eq:s/n}
%with the second part of Lemma~\ref{Lem:NormalCount}. 
\end{proof}

\begin{remark}
%\item[(ii)] 
Similarly, one can show that
the number $s_{p^2}(\Gamma)$ does not attain the values
  $2, 3, \ldots, p-1$ modulo $p$. More precisely, if\/ $\Gamma$ contains
  a subgroup of index $p^2,$ then 
\begin{equation}
\label{Eq:FinGenIndp^2Cong}
s_{p^2}(\Gamma) \equiv \begin{cases} 0~(\mathrm{mod}~p),& 
\text{if\/ }\bar{r}_p(\Gamma) +
  \bar{r}_\infty(\Gamma)=0, \mbox{ or }e_1^{(p)} = \bar{r}_p(\Gamma) =
  1 \mbox{ and } \, \bar{r}_\infty(\Gamma)=0,\\[1mm] 
1~(\mathrm{mod}~p),& 
\mbox{otherwise.} \end{cases}
\end{equation}
\end{remark}

We give two illustrations for Proposition~\ref{Prop:FinGenIndexp} 
from finite group theory.

\begin{corollary} \label{cor:pm1}
Let\/ $G$ be a finite non-Abelian simple group, and let $p^m$ be a
non-trivial prime power dividing the order of\/ $G$. Then $s_{p^m}(G)
\equiv 0~(\mathrm{mod}~p)$. 
\end{corollary}

\begin{proof}
We have $1 < p^m < \vert G\vert$, since otherwise $1<\zeta_1(G) < G$,
a contradiction. Hence, $G$ does not contain a normal subgroup of
index $p^m$, and our assertion  follows from 
the last assertion of Proposition~\ref{Prop:FinGenIndexp}.  
\end{proof}

\begin{corollary} \label{cor:pm2}
Let\/ $p^m$ be a prime power with $p^m\geq3,$ let $n$ be positive
integer, and suppose that $p^m \mid n!$. Then $s_{p^m}(S_n) \equiv
0~(\mathrm{mod}~p)$. 
\end{corollary}

\begin{proof}
For $n=1,2$ the assertion is empty (thus holds trivially), for $n=3,4$
it holds by inspection. For $n\geq5$, the only non-trivial normal
subgroup of $S_n$ is $A_n$, which has index $2$; in particular, $S_n$
does not contain a normal subgroup of index $p^m$. Our claim follows
again from the last assertion of Proposition~\ref{Prop:FinGenIndexp}. 
\end{proof}

\section{$p$-Divisibility of homomorphism numbers 
of finitely generated groups}
\label{sec:appl}

\noindent In this section, we present results on the
$p$-divisibility of homomorphism numbers 
for various classes of finitely generated groups~$\Gamma$.
Our first result, Theorem~\ref{thm:fingen}, in particular says that
there exists a sharp dividing line for the $p$-divisibility
of the sequence $\big(h_n(\Gamma)\big)_{n\ge0}$: either there is
``no increase" in $p$-divisibility, or there is a considerable $p$-part
in $h_n(\Gamma)$ which tends to infinity as $n\rightarrow\infty$.
In this context,
Proposition~\ref{Prop:FinGenIndexp} shows that the dividing line
is given by $s_p(\Gamma)\equiv 1$~(mod~$p$) as opposed to
$s_p(\Gamma)\equiv 0$~(mod~$p$). This dichotomy is indeed
what is needed in the proof of Theorem~\ref{thm:fingen} 
in order to be able to apply an appropriate 
abstract result (namely Corollary~\ref{cor:divline}).

Finite $p$-groups $G$ always have the property
that $s_p(G)\equiv 1$~(mod~$p$). Thus, unbounded growth of $p$-divisibility
must be expected for the homomorphism numbers of such groups.
Theorem~\ref{thm:pgroup} shows that a
growth estimate for $h_n(G)$ can be given which is even better than
the one in Theorem~\ref{thm:fingen}. Indeed,
the bound given there is the same as the one in
\eqref{Eq:CpVal} for the cyclic group~$C_p$. 
Moreover, if $p$ is odd and $G$ is not cyclic, a further improvement
is possible, as Theorem~\ref{thm:Kul} shows.

As an example of a non-nilpotent group of mixed order, in
Theorem~\ref{cor:D} we consider the dihedral group of order $2m$,
which turns out to have the property that the $2$-divisibility of
its homomorphism numbers can be bounded below by bounds which are
at least as good as the one for $C_2$ given in \eqref{eq:hnC2},
and a better bound if $4\mid m$.

In sharp contrast to these results, we show that a finite non-Abelian simple
group $G$ satisfies $h_n(G)\equiv 1$~(mod~$p$) for
all~$n$; see Corollary~\ref{Cor:FiniteSimplep}. 

\begin{theorem} \label{thm:fingen}
Let\/ $\Gamma$ be a finitely generated group, and let $p$ be a prime number.
If $\Gamma$ contains a subgroup of index~$p$ and
$\bar{r}_p(\Gamma) + \bar{r}_\infty(\Gamma)
>0,$ then, for $p\ge5,$ we have
\begin{equation} \label{eq:hdivl1}
v_p\big(h_n(\Gamma)\big)\ge 
\sum_{s \ge 1}\left(\fl{\frac {n} {p^{s}}}
-\fl{\frac {n} {2p^{s}}}\right),
\end{equation}
while, for $p=3,$ we have
\begin{equation} \label{eq:hdivl3}
v_3\big(h_n(\Gamma)\big)\ge 
\sum_{s \ge 1}\left(\fl{\frac {n} {3^{s}}}
-\fl{\frac {n} {2\cdot3^{s}}}\right)-\fl{\frac {n} {18}},
\end{equation}
and, for $p=2,$ we have
\begin{equation} \label{eq:hdivl2}
v_2\big(h_n(\Gamma)\big)\ge \fl{\frac {n} {2}}-\fl{\frac {n} {4}},
\end{equation}
for all $n$.
In all other cases, for each $N$
there exists some $n>N$ such that $h_n(\Gamma)$ is not divisible by $p$. 
\end{theorem}

\begin{proof}
This follows by combining Proposition~\ref{Prop:FinGenIndexp} with 
Relation \eqref{eq:HS}  
and Corollary~\ref{cor:divline}.
\end{proof}

\begin{theorem} \label{thm:pgroup}
Let $p$ be a prime number, and let $G$ be a non-trivial finite $p$-group.
Then
\begin{equation} \label{eq:hpdiv}
v_p\big(h_n(G)\big)\ge 
\fl{\frac {n} {p}}
-\fl{\frac {n} {p^{2}}}
\end{equation}
for all $n$.
\end{theorem}

\begin{proof}
By Frobenius' generalisation of Sylow's third theorem in \cite{Frob1},  
we have $s_{p^i}(G)\equiv1$~(mod~$p$) for all $i$ such that
$p^i\le \vert G\vert$, while $s_n(G)=0$ for all $n$ different
from a power of $p$. The claim now
follows from Relation \eqref{eq:HS} plus
Corollary~\ref{thm:pdiv10} with $l=2$.
\end{proof}

\begin{theorem} \label{thm:Kul}
Let $p$ be a prime number, and
let $G$ be a non-cyclic $p$-group of odd order.
Then
\begin{equation} \label{eq:vpKul}
v_p\big(h_n(G)\big)\ge
\fl{\frac {n} {p}}
+\fl{\frac {n} {p^2}}
-2\fl{\frac {n} {p^3}}
\end{equation}
for all $n$.
\end{theorem}

\begin{proof}
Let $p^m$ be the order of $G$.
By \cite{KulaAA} (see also \cite[Theorem~1.52]{HallAA}),
we know that $s_{p^i}\equiv 1+p$~(mod~$p^2$) for $1\le i\le m-1$.
Thus, remembering again \eqref{eq:HS}, we see that the assumptions
of Corollary~\ref{thm:pdiv10} with $l=3$ are satisfied.
The claim now follows from \eqref{eq:vpH10}.
\end{proof}

Next we discuss the family of finite dihedral groups.

\begin{proposition}
\label{Prop:FiniteDihedral}
Let $D_m$
be the dihedral group of order $2m,$ and let $d$ be a divisor of
$2m$. Then
\[
s_{d}(G) = \begin{cases} d,& \text{if\/ $d$ is odd,}\\[1mm]
1+d,& \text{if\/ $d$ is even}.\end{cases}
\]
\end{proposition}

\begin{proof}
We regard $D_m$ as a real reflection group acting on $\mathbb R^2$
by rotations through angles that are multiples of $\pi/m$ 
and reflections in lines that have angles with the $x$-axis that
are also multiples of $\pi/m$. 
Subgroups of $D_m$ are either cyclic or themselves dihedral groups.
For each even divisor $d$ of $2m$ there exists exactly one cyclic
subgroup of $D_m$ of index~$d$. On the other hand, a dihedral group
can be given in terms of two reflections which generate it.
It is not difficult to see that a unique way to encode a dihedral
subgroup of $D_m$ of index $d$ is in terms of two reflections with respect to
the lines $l_1$ and $l_2$, the first having an angle of $\al\pi/m$
with the $x$-axis, the second having an angle of $(\al+d)\pi/m$ with
the $x$-axis, and $0\le \al<d$. There are $d$ possibilities to choose
$\al$. This implies the assertion of the proposition.
\end{proof}

\begin{theorem} \label{cor:D}
Let $D_m$ be the dihedral group of order~$2m$. If $p$ is a prime
number different from~$2,$ then $h_n(D_m)$ is infinitely often not
divisible by~$p,$ while we have
\begin{equation} \label{eq:D}
v_2\big(h_n(D_m)\big)\ge 
\begin{cases} 
\fl{\frac {n} {2}},&\text{if $4\mid m$,}\\[1mm]
\fl{\frac {n} {2}}-\fl{\frac {n} {4}},
&\text{if $4\nmid m$.}
\end{cases}
\end{equation}
\end{theorem}

\begin{proof}
First consider an odd prime number $p$. In that case,
Proposition~\ref{Prop:FiniteDihedral} says that $s_p(D_m)=p$
if $p\mid m$ and $s_p(D_m)=0$ if not. In either case, we have
$s_p(D_m)\equiv0$~(mod~$p$). The assertion on $h_n(D_m)$ in
that case then follows directly from Corollary~\ref{cor:divline}.

In order to establish the second part of the theorem, we observe that,
by Proposition~\ref{Prop:FiniteDihedral}, we have
\begin{equation*}
s_1(D_m)=1,\quad 
s_2(D_m)=3,\quad 
%s_3(D_m)=\begin{cases} 
%0,&\text{if\/ $3\nmid m$,}\\
%3,&\text{if\/ $3\mid m$,}\\
%\end{cases}
s_4(D_m)=\begin{cases} 
0,&\text{if\/ $4\nmid m$,}\\
5,&\text{if\/ $4\mid m$,}\\
\end{cases}\quad 
s_8(D_m)=\begin{cases} 
0,&\text{if\/ $8\nmid m$,}\\
9,&\text{if\/ $8\mid m$.}\\
\end{cases}
\end{equation*}
Consequently, we have $s_2(D_m)-s_1(D_m)=2\equiv0$~(mod~$2$) and
$s_4(D_m)-s_2(D_m)=2$ if $4\mid m$, while $s_4(D_m)-s_2(D_m)=-3$ 
if $4\nmid m$.
Thus, in the first case, we may apply 
Corollary~\ref{thm:pdiv1u} with $l=2$
and $m=1$, while in the second case we have to
choose $m=0$. The bound on the $2$-divisibility of $h_n(D_m)$ then
follows from \eqref{eq:sdiffu}.
\end{proof}

The following proposition prepares for the result in
Corollary~\ref{Cor:FiniteSimplep} on the residue class modulo~$p$
of the number of permutation representations of a finite non-Abelian
simple group.
Given a finitely generated group $\Gamma$,
we write $T_n(\Gamma)$ for the set of all transitive permutation
representations of $\Gamma$ of degree $n$. 
It is well-known (see e.g.\ \cite[Prop.~3]{KM}) that
\begin{equation}
\label{Eq:TransReps}
\vert T_n(\Gamma)\vert = (n-1)!\,s_n(\Gamma),\quad\quad 
\text{for } n\geq1.
\end{equation}

\begin{proposition}
\label{Prop:FiniteSimpleT}
Suppose that $G$ is a finite non-Abelian simple group. Then, for each
$n\geq2,$ we have $\vert T_n(G)\vert \equiv 0
~(\mathrm{mod}~\vert G\vert)$. 
\end{proposition}

\begin{proof}
It is easy to check that $G$ acts from the left on the set $T_n(G)$ by
conjugation, that is, 
\[
g\cdot \varphi = \varphi \circ \iota_g,\quad \quad 
\text{for }g\in G, \,\varphi\in T_n(G). 
\]
Now suppose that $g\cdot \varphi = \varphi$ for some $g\in G$ and
$\varphi\in T_n(G)$. Then we have  
\[
\varphi(g^{-1}hg) = \varphi(h),\quad\quad 
\text{for } h\in G,
\] 
or, equivalently,
\[
[g, h]\in \mathrm{ker}(\varphi),\quad\quad 
\text{for } h\in G.
\]
However, as $\varphi$ is transitive on the set $\{1, 2, \ldots, n\}$
and $n\geq2$ by assumption, we have $\mathrm{ker}(\varphi) \neq G$,
hence, by simplicity of $G$, $\mathrm{ker}(\varphi) = 1$. We thus
conclude that  
\[
[g, h] = 1,\quad\quad 
\text{for } h\in G,
\] 
so that $g\in \zeta_1(G)$. Using again simplicity of $G$, plus the
fact that $G$ is non-Abelian, we find that $\zeta_1(G)=1$, thus
$g=1$. Consequently, the action of $G$ on $T_n(G)$ is free, whence our
claim. 
\end{proof}

\begin{corollary}
\label{Cor:FiniteSimplep}
Let\/ $G$ be a finite non-Abelian simple group, and let $p^m$ be a prime power
dividing  the order of $G$. Then 
\begin{equation}
\label{Eq:FiniteSimplep}
s_n(G) \equiv 0~(\mathrm{mod}~p^m),\quad 2\leq n\leq p.
\end{equation}
\end{corollary}

\begin{proof}
Since $n\leq p$ by assumption, $p \nmid (n-1)!$, so that $(n-1)!$ is
invertible modulo~$p^m$. Combining \eqref{Eq:TransReps} with
Proposition~\ref{Prop:FiniteSimpleT}, we get, for $2\leq n\leq p$,
that 
\[
s_n(G) \equiv ((n-1)!)^{-1}\, \vert T_n(G)\vert \equiv
0~(\mathrm{mod}~p^m), 
\]
as desired.
\end{proof}

From the above corollary, we see in particular that 
$s_p(G)\equiv 0$~(mod~$p$). Thus,
by Corollary~\ref{cor:divline}, we conclude that there is no increasing
$p$-divisibility for $h_n(G)$, in the sense that, for each positive
integer $N$, we can find $n>N$ such that $h_n(G)$ is not divisible by $p$.
We now show that the above corollary allows for a sharpening of the last
conclusion: $h_n(G)$ is {\it never} divisible by $p$, and it is in
fact congruent to~$1$ modulo~$p$.

\begin{corollary}
Let\/ $G$ be a finite non-Abelian simple group, and let $p$ be a prime divisor
of the order of $G$. Then we have 
\[
h_n(G) \equiv 1~(\mathrm{mod}~p),\quad n\geq1.
\]
\end{corollary}

\begin{proof}
Combining the fact that $s_n(G) \equiv 0~(\mathrm{mod}~p)$ for
$2\leq n\leq p$, coming from Corollary~\ref{Cor:FiniteSimplep}, with
the recurrence in \eqref{Eq:HallTransform}, we find that 
\[
h_n(G) = \sum_{k=1}^n ( n-k+1)_{k-1} s_k(G) h_{n-k}(G)
\equiv h_{n-1}(G)~(\mathrm{mod}~p),\quad n\geq1.
\]
Since $h_0(G) = 1$, the result follows.
\end{proof}

\section{$p$-Divisibility of homomorphism numbers for 
finite Abelian $p$-groups}
\label{sec:Abel}

\noindent In this section, we provide tight bounds on the divisibility
by powers of a prime~$p$ of the number of permutation representations 
of finite Abelian $p$-groups. Theorems~\ref{thm:rang} and
\ref{thm:rang2} below
refine the results of Katsurada, Takegahara and Yoshida
\cite[Theorems~1.2--1.4]{KaTYAA} for rank~$1$ and $2$
by adding a periodicity assertion for quotients, 
while, at the same time, generalising them to arbitrary rank.

\begin{theorem} \label{thm:rang}
Let\/ $G=C_{p^{a_1}}\times C_{p^{a_2}}\times\dots \times
C_{p^{a_r}}$ with $a_1\ge a_2\ge \dots \ge a_r$. 

{\em (i)} If $a_1> a_2+\dots+a_r,$ then
\begin{equation} \label{eq:hnranga} 
v_p\big(h_n(G)\big)\ge
\sum_{s=1}^{a_1}\fl{\frac {n} {p^s}}
-(a_1-a_2-\dots-a_r)\fl{\frac {n} {p^{a_1+1}}}.
\end{equation}
Moreover, the quotient
\begin{equation*} %\label{eq:quotr3}
Q_n(G)=\frac {h_n(G)} 
{p^{e_p(n;a_1,\dots,a_r)}},
\end{equation*}
where $e_p(n;a_1,\dots,a_r)$ 
denotes the right-hand side of \eqref{eq:hnranga}$,$
satisfies
\begin{equation} \label{eq:Qrek1ra} 
Q_n(G)\equiv (-1)^{a_1}Q_{n-p^{a_1+1}}(G) 
\quad \text{\em (mod $p$)}.
\end{equation}
In particular, the bound in \eqref{eq:hnranga} is tight for all $n$ which are
divisible by $p^{a_1+1}$.

{\em (ii)} If $a_1\le a_2+\dots+a_r$ and $a_1+a_2+\dots+a_r$ is even, then
\begin{equation} \label{eq:hnrangb} 
v_p\big(h_n(G)\big)\ge
\sum_{s=1}^{A_1}\fl{\frac {n} {p^s}},
\end{equation}
where $A_1=(a_1+a_2+\dots+a_r)/2$.
Moreover, if $p>2,$ the quotient
\begin{equation*} %\label{eq:quotr3}
Q_n(G)=\frac {h_n(G)} 
{p^{e_p(n;A_1)}},
\end{equation*}
where $e_p(n;A_1)$ denotes the right-hand side of \eqref{eq:hnrangb}$,$
satisfies
\begin{equation} \label{eq:Qrek1rb} 
Q_n(G)\equiv (-1)^{A_1}Q_{n-p^{A_1+1}}(G) 
\quad \text{\em (mod $p$)}.
\end{equation}
In particular, the bound in \eqref{eq:hnrangb} is tight for all $n$ which are
divisible by $p^{A_1+1},$ except if $p=2$.

{\em (iii)} If $a_1\le a_2+\dots+a_r$ and $a_1+a_2+\dots+a_r$ is odd, then
\begin{equation} \label{eq:hnrangc} 
v_p\big(h_n(G)\big)\ge
\sum_{s=1}^{A_2}\fl{\frac {n} {p^s}}-\fl{\frac {n} {p^{A_2+1}}},
\end{equation}
where $A_2=(a_1+a_2+\dots+a_r+1)/2$.
Moreover, the quotient
\begin{equation*} %\label{eq:quotr3}
Q_n(G)=\frac {h_n(G)} 
{p^{e_p(n;A_2)}},
\end{equation*}
where $e_p(n;A_2)$ denotes the right-hand side of \eqref{eq:hnrangc}$,$
satisfies
\begin{equation} \label{eq:Qrek1rc} 
Q_n(G)\equiv (-1)^{A_2}Q_{n-p^{A_2+1}}(G) 
\quad \text{\em (mod $p$)}.
\end{equation}
In particular, the bound in \eqref{eq:hnrangc} is tight for all $n$ which are
divisible by $p^{A_2+1}$.
\end{theorem}

\begin{proof}
By a short and elegant computation, Butler \cite[display on top of
p.~773]{ButlAA} proved that the difference of ``successive" subgroup
numbers in a finite Abelian $p$-group of type $\al
=(a_1,a_2,\dots,a_r)$ essentially equals a specialised Kostka--Foulkes
polynomial. To be precise, taking into account the symmetry relation
(cf.\ e.g.\ \cite[p.~181, Statement~(1.5)]{MacdAC})
\begin{equation} \label{eq:sym3} 
s_{p^i}=s_{p^{2A_1-i}},
\end{equation} 
Butler found that
\begin{equation} \label{eq:sK} 
s_{p^i}(G)-s_{p^{i-1}}(G)=
p^{n(\al)}K_{(2A_1-i,i),\al}(p^{-1}),
\quad \text{for $i\le A_1$},
\end{equation}
where $K_{\la,\mu}(t)$ denotes the Kostka--Foulkes polynomial indexed by
partitions $\la$ and $\mu$ (see \cite[Ch.~III, Sec.~6]{MacdAC} for the
definition), and where $n(\al)=\sum_{i=1}^r(i-1)a_i$.
It is known (cf.\ \cite[Statement~(6.5)(ii) on p.~243]{MacdAC}) 
that $K_{\la,\mu}(t)$ vanishes if $\mu$ is not less than or equal to $\la$
in dominance order (see \cite[p.~7]{MacdAC} for the definition),
and that it is a monic polynomial of degree
$n(\la)-n(\mu)$ otherwise. 
We should observe that $\al$ is not less
than or equal to $(2A_1-i,i)$ if, and only if, $2A_1-i<a_1$, that is,
$i>a_2+\dots+a_r$. 
If we use all this and the simple fact that
$n\big((2A_1-i,i)\big)=i$ to rewrite \eqref{eq:sK},
$$
s_{p^i}(G)-s_{p^{i-1}}(G)=
p^{i}\big(p^{n(\al)-n((2A_1-i,i))}K_{(n-i,i),\al}(p^{-1})\big),
$$
and combine this with the symmetry relation \eqref{eq:sym3},
then it follows immediately that 
\begin{equation} \label{eq:diff1}
v_p\big(s_{p^i}(G)-s_{p^{i-1}}(G)\big)=
\begin{cases}
i,&\text{for }0\le i\le \min\left\{2A_1-a_1,A_1\right\},\\
2A_1-i+1,&\text{for }\max\left\{a_1,A_1\right\}+1\le i\le 2A_1+1,
\end{cases}
\end{equation}
and
\begin{multline} \label{eq:diff1b}
s_{p^i}(G)-s_{p^{i-1}}(G)=0,\quad 
\text{for } 2A_1-a_1+1\le i\le a_1,\\
\text{and for }i=A_2\text{ if $a_1+\dots+a_r$ is odd.}
\end{multline}
Moreover, we have
\begin{align} 
\notag
s_{p^i}(G)-s_{p^{i-1}}(G)&=-p^{2A_1-i+1}+\text{higher degree terms},\\
\notag
&\text{for } i=a_1+1\text{ if }a_1>a_2+\dots+a_r,\\
\notag
&\text{for }i=A_1+1\text{ if $a_1\le a_2+\dots+a_r$ and
$a_1+\dots+a_r$ is even, }\\
&\text{and for }i=A_2+1\text{ if $a_1\le a_2+\dots+a_r$ and
$a_1+\dots+a_r$ is odd.}
\label{eq:diff1c}
\end{align}

If we let $a_1> a_2+\dots+a_r$, then \eqref{eq:diff1}--\eqref{eq:diff1c}
imply that we may apply 
Corollary~\ref{thm:pdiv1u} with $l=a_1+1$ and $m=2A_1-a_1+1=a_2+\dots+a_r$.
Using Legendre's formula \cite[p.~10]{LegeAA} for the $p$-adic
valuation of a factorial again, we obtain
the bound \eqref{eq:hnranga}, as well as the congruence
\eqref{eq:Qrek1ra} and the corresponding tightness assertion.

Similarly, 
if $a_1\le a_2+\dots+a_r$ and $a_1+a_2+\dots+a_r$ is even, 
then \eqref{eq:diff1} and \eqref{eq:diff1c} imply that
we may apply Corollary~\ref{thm:pdiv1u} with
$l=A_1+1$ and $m=A_1$,
except if $p=2$. Indeed, while \eqref{eq:pS1u} and \eqref{eq:p^mu}
are satisfied for the choice of $l=A_1+1$ and $m=A_1$,
the inequality \eqref{eq:sdiffu} is satisfied only if we do not
have $p=2$. Namely, if we choose $e=A_1+2$ in the latter case,
the left-hand side of \eqref{eq:sdiffu} equals $A_1-1$,
while the right-hand
side gives $-2-1+(A_1+2)+1=A_1$, contradicting \eqref{eq:sdiffu}. 
However, if we exclude the case where $p=2$, then
Corollary~\ref{thm:pdiv1u} 
together with Legendre's formula \cite[p.~10]{LegeAA} imply
the bound \eqref{eq:hnrangb}, 
as well as the congruence \eqref{eq:Qrek1rb} and
corresponding tightness assertion.
In the exceptional case where $p=2$, the stronger
Theorem~\ref{thm:rang2} below fills the hole.

Finally, if $a_1\le a_2+\dots+a_r$ and $a_1+a_2+\dots+a_r$ is odd, 
then \eqref{eq:diff1}--\eqref{eq:diff1c} imply that
we may apply Corollary~\ref{thm:pdiv1u} with 
$l=A_2+1$ and $m=A_2-1$.
Using Legendre's formula \cite[p.~10]{LegeAA}, 
we then obtain the bound \eqref{eq:hnrangc},
as well as the congruence \eqref{eq:Qrek1rc} and
corresponding tightness assertion.

This completes the proof of the theorem.
\end{proof}

In the exceptional case where $p=2$, $a_1\le a_2+\dots+a_r$, 
and $a_1+a_2+\dots+a_r$ is
even, we may instead apply Theorem~\ref{thm:pdiv12} to obtain
a stronger $2$-divisibility result. Since the proof is
straightforward from \eqref{eq:diff1}--\eqref{eq:diff1c}, 
we content ourselves with the statement of the result.

\begin{theorem} \label{thm:rang2}
Let\/ $G=C_{2^{a_1}}\times C_{2^{a_2}}\times\dots \times
C_{2^{a_r}}$ with $a_1\ge a_2\ge \dots \ge a_r,$
$a_1\le a_2+\dots+a_r$, and $a_1+a_2+\dots+a_r$ being even.
Then
\begin{equation} \label{eq:hnrang2} 
v_2\big(h_n(G)\big)\ge \sum_{s=1}^{A_1}\fl{\frac {n} {2^s}}
+\fl{\frac {n} {2^{A_1+2}}}
-\fl{\frac {n} {2^{A_1+3}}},
\end{equation}
where $A_1=(a_1+a_2+\dots+a_r)/2$.
Moreover, the quotient
\begin{equation*} %\label{eq:quot2}
Q_n(G)=\frac {h_n(G)} 
{2^{e_2(n;A_1)}},
\end{equation*}
where $e_2(n;A_1)$ denotes the right-hand side of \eqref{eq:hnrang2}$,$
satisfies
\begin{equation} \label{eq:Qrek2} 
Q_n(G)\equiv Q_{n-2^{A_1+3}}(G) 
\quad \text{\em (mod $2$)}.
\end{equation}
The bound is tight for all $n$ congruent to 
$0,$ $2^{A_1+1},$ and $2^{A_1+2}$ modulo $2^{A_1+3}$.
\end{theorem}

\section{Periodicity of 
Subgroup numbers for free products of finite Abelian groups}
\label{sec:per}

In \cite{GrNeAA}, Grady and Newman used their $p$-divisibility
result for the homomorphism number $h_n(C_p)$ of $C_p$ mentioned in the
introduction to demonstrate ultimate periodicity modulo~$p$ for the subgroup
numbers of free powers of $C_p$. Armed with our much more general
$p$-divisibility results from the previous section, and using
the same approach, we may now derive ultimate periodicity modulo~$p$ 
for much larger classes of free products.

\begin{theorem} \label{thm:per}
Let $p$ be a prime number. Furthermore,
let $\Gamma_0$ be a finitely generated group, and let
$G_1,G_2,\dots$ be non-trivial finite Abelian $p$-groups. 
In each of the following cases,
the arithmetic function $s_n(\Gamma)$ forms
an ultimately periodic sequence modulo~$p$:

\begin{enumerate} 
\item $\Gamma=\Gamma_0*G_1*G_2$ and $p\ge5$.
\item $\Gamma=\Gamma_0*G_1*G_2,$ $p=3,$ and not both $G_1$ and $G_2$
are isomorphic to $C_3$.
\item $\Gamma=\Gamma_0*G_1*G_2*G_3$ and $p=3$.
\item $\Gamma=\Gamma_0*G_1*G_2,$ $p=2,$ and both
$G_1$ and $G_2$ are not isomorphic to $C_2$.
\item $\Gamma=\Gamma_0*C_2*G_1,$ $p=2,$ and $G_1$ is not isomorphic 
to $C_2,C_4,C_8,C_2\times C_2,\break C_2\times C_2\times C_2,$ or
$C_4\times C_2$.
\item $\Gamma=\Gamma_0*G_1*G_2*G_3,$ $p=2,$ and not all
of $G_1,G_2,G_3$ are isomorphic to $C_2$.
\item $\Gamma=\Gamma_0*C_2*C_2*C_2*C_2$ and $p=2$.
\end{enumerate}

\end{theorem}

\begin{proof}
We know that the sequence $\big(s_n(\Gamma)\big)_{n\ge1}$ satisfies the
recurrence \eqref{Eq:HallTransform}. By dividing both sides of the 
recurrence by $(n-1)!$, we obtain the equivalent form
\begin{equation} \label{eq:rec2} 
\frac {nh_n(\Gamma)} {n!} = \sum_{k=1}^n  s_k(\Gamma)\,
\frac {h_{n-k}(\Gamma)} {(n-k)!},\quad \quad \text{for }n\geq1. 
\end{equation}
Following Grady and Newman \cite{GrNeAA}, our strategy consists in
showing that $v_p\big(h_n(\Gamma)/n!\big)>0$ for almost all $n$ in the cases
(1)--(7). Given that $p$-divisibility property, the recurrence
\eqref{eq:rec2} reduces to a finite-length linear recurrence with
constant coefficients
for the sequence $\big(s_n(\Gamma)\big)_{n\ge1}$ when
considered modulo~$p$. It then follows (see e.g.\ \cite[Ch.~8]{LiNie})
that $\big(s_n(\Gamma)\big)_{n\ge1}$ is an ultimately periodic
sequence modulo~$p$.

\medskip
(1) We have
$$
h_n(\Gamma)=h_n(\Gamma_0)h_n(G_1)h_n(G_2),
$$
and hence
\begin{equation} \label{eq:hG1} 
v_p\left(h_n(\Gamma)\right)\ge
v_p\left(h_n(G_1)\right)+v_p\left(h_n(G_2)\right).
\end{equation}
We have to show that 
\begin{equation} \label{eq:h>n!1} 
v_p\left(h_n(\Gamma)\right)>v_p(n!)
=\sum_{s\ge1}\fl{\frac {n} {p^s}}
\end{equation}
for all sufficiently large $n$.
By our $p$-divisibility results in Theorem~\ref{thm:rang}, we see
that
$$
v_p\left(h_n(G)\right)\ge \fl{\frac {n} {p}}-\fl{\frac {n} {p^2}}
$$
for any finite Abelian $p$-group $G$. Hence, 
\begin{equation*} 
v_p\left(h_n(G_1)\right)+v_p\left(h_n(G_2)\right)
\ge 2\left(\fl{\frac {n} {p}}-\fl{\frac {n} {p^2}}\right).
\end{equation*}
Now let $n=n_2p^2+n_1p+n_0$ with $0\le n_0,n_1<p$.
Then the above inequality becomes
\begin{equation} \label{eq:(1)a}  
v_p\left(h_n(G_1)\right)+v_p\left(h_n(G_2)\right)
\ge 2\left(n_2p+n_1-n_2\right),
\end{equation}
while
\begin{align} \notag
v_p(n!)&=(n_2p+n_1)+n_2+\sum_{s\ge3}\fl{\frac {n} {p^s}}\\
\notag
&=n_2(p+1)+n_1+\sum_{s\ge1}\fl{\frac {n_2} {p^s}}\\
\notag
&\le n_2(p+1)+n_1+\sum_{s\ge1}{\frac {n_2} {p^s}}\\
&\le n_2\left(p+1+\frac {1} {p-1}\right)+n_1.
\label{eq:(1)b}  
\end{align}
Since 
\begin{equation} \label{eq:pungl} 
2(p-1)> p+1+\frac {1} {p-1}
\end{equation}
for all $p\ge5$, a combination of 
\eqref{eq:hG1}, \eqref{eq:(1)a} and
\eqref{eq:(1)b} establishes \eqref{eq:h>n!1} as long as
$n_2>0$, that is, for $n\ge p^2$.

\medskip
(2) We proceed as in the proof of Item~(1). Without loss of
generality, let us assume that $G_2$ is not isomorphic to $C_3$.
By our $3$-divisibility results in Theorem~\ref{thm:rang}, it then
follows that
$$
v_3\left(h_n(G_2)\right)\ge 
\fl{\frac {n} {3}}+\fl{\frac {n} {9}}-2\fl{\frac {n} {27}}.
$$
Hence,
\begin{equation*} 
v_3\left(h_n(G_1)\right)+v_3\left(h_n(G_2)\right)
\ge 2\left(\fl{\frac {n} {3}}-\fl{\frac {n} {27}}\right).
\end{equation*}
Now let $n=27n_3+9n_2+3n_1+n_0$ with $0\le n_0,n_1,n_2<3$.
Then the above inequality becomes
\begin{align} \notag
v_3\left(h_n(G_1)\right)+v_3\left(h_n(G_2)\right)
&\ge 2\left(9n_3+3n_2+n_1-n_3\right)\\
&\ge 16n_3+6n_2+2n_1,
\label{eq:(2)a}  
\end{align}
while
\begin{align} \notag
v_3(n!)&=(9n_3+3n_2+n_1)+(3n_3+n_2)+n_3+\sum_{s\ge4}\fl{\frac {n} {3^s}}\\
\notag
&=13n_3+4n_2+n_1+\sum_{s\ge1}\fl{\frac {n_3} {3^s}}\\
\notag
&\le 13n_3+4n_2+n_1+\sum_{s\ge1}{\frac {n_3} {3^s}}\\
&\le \frac {27} {2}n_3+4n_2+n_1.
\label{eq:(2)b}  
\end{align}
A combination of 
\eqref{eq:hG1}, \eqref{eq:(2)a} and
\eqref{eq:(2)b} establishes \eqref{eq:h>n!1} as long as
$n_3>0$, that is, for $n\ge 27$.

\medskip
(3) This is completely analogous to the proof of Item~(1).
The only difference is that, instead of \eqref{eq:pungl},
here we rely on
$$
3(p-1)> p+1+\frac {1} {p-1},
$$
which is valid for all $p\ge3$, so in particular for $p=3$.

\medskip
(4) 
By our $2$-divisibility results in Theorems~\ref{thm:rang}
and \ref{thm:rang2}, one can see that
$$
v_2\big(h_n(G_1)\big)\ge 
\fl{\frac {n} {2}}+\fl{\frac {n} {4}}.
$$
Hence,
\begin{equation*} 
v_2\big(h_n(C_2)\big)+v_2\big(h_n(G_1)\big)
\ge 2\fl{\frac {n} {2}}.
\end{equation*}
Now let $n=2n_1+n_0$ with $0\le n_0<2$.
Then the above inequality becomes
\begin{equation}  
v_2\left(h_n(C_2)\right)+v_2\left(h_n(G_1)\right)
\ge 2n_1,
\label{eq:(4)a}  
\end{equation}
while
\begin{align} \notag
v_2(n!)&=n_1+\sum_{s\ge2}\fl{\frac {n} {2^s}}\\
\notag
&=n_1+\sum_{s\ge1}\fl{\frac {n_1} {2^s}}\\
&< 2n_1,
\label{eq:(4)b}  
\end{align}
as long as $n_1>0$.
A combination of 
\eqref{eq:hG1}, \eqref{eq:(4)a} and
\eqref{eq:(4)b} establishes \eqref{eq:h>n!1} as long as
$n_1>0$, that is, for $n\ge 2$.

\medskip
(5) Again, we proceed as in the proof of Item~(1). 
By our $2$-divisibility results in Theorems~\ref{thm:rang}
and \ref{thm:rang2}, it follows that
$$
v_2\big(h_n(G)\big)\ge 
\fl{\frac {n} {2}}+\fl{\frac {n} {4}}-2\fl{\frac {n} {8}}
$$
for any finite Abelian $2$-group not isomorphic to $C_2$.
Hence,
\begin{equation*} 
v_2\big(h_n(G_1)\big)+v_2\big(h_n(G_2)\big)
\ge 2\fl{\frac {n} {2}}+2\fl{\frac {n} {4}}-4\fl{\frac {n} {8}}.
\end{equation*}
Now let $n=8n_3+4n_2+2n_1+n_0$ with $0\le n_0,n_1,n_2<2$.
Then the above inequality becomes
\begin{align} \notag
v_2\big(h_n(G_1)\big)+v_2\big(h_n(G_2)\big)
&\ge 2(4n_3+2n_2+n_1)+2(2n_3+n_2)-4n_3\\
&\ge 8n_3+6n_2+2n_1,
\label{eq:(5)a}  
\end{align}
while
\begin{align} \notag
v_2(n!)&=(4n_3+2n_2+n_1)+(2n_3+n_2)+n_3+\sum_{s\ge4}\fl{\frac {n} {2^s}}\\
\notag
&=7n_3+3n_2+n_1+\sum_{s\ge1}\fl{\frac {n_3} {2^s}}\\
&< 8n_3+3n_2+n_1,
\label{eq:(5)b}  
\end{align}
as long as $n_3>0$.
A combination of 
\eqref{eq:hG1}, \eqref{eq:(5)a} and
\eqref{eq:(5)b} establishes \eqref{eq:h>n!1} as long as
$n_3>0$, that is, for $n\ge 8$.

\medskip
(6) We proceed as in the proof of Item~(1). 
Here, we have to show that
\begin{equation} \label{eq:h>n!4} 
v_2\big(h_n(G_1)\big)+v_2\big(h_n(G_2)\big)
+v_2\big(h_n(G_3)\big)
>\sum_{s\ge1}\fl{\frac {n} {2^s}}
\end{equation}
for all sufficiently large $n$.
Without loss of
generality, we assume that $G_3$ is not isomorphic to $C_2$.
By our $2$-divisibility results in Theorems~\ref{thm:rang}
and \ref{thm:rang2}, it then follows that
$$
v_2\big(h_n(G_3)\big)\ge 
\fl{\frac {n} {2}}+\fl{\frac {n} {4}}-2\fl{\frac {n} {8}}.
$$
Hence,
\begin{equation*} 
v_2\big(h_n(G_1)\big)+v_2\big(h_n(G_2)\big)
+v_2\big(h_n(G_3)\big)
\ge 3\fl{\frac {n} {2}}-\fl{\frac {n} {4}}-2\fl{\frac {n} {8}}.
\end{equation*}
Now let $n=8n_3+4n_2+2n_1+n_0$ with $0\le n_0,n_1,n_2<2$.
Then the above inequality becomes
\begin{align} \notag
v_2\big(h_n(G_1)\big)+v_2\big(h_n(G_2)\big)
&+v_2\big(h_n(G_3)\big)\\
\notag
&\ge 3(4n_3+2n_2+n_1)-(2n_3+n_2)-2n_3\\
&\ge 8n_3+5n_2+3n_1,
\label{eq:(6)a}  
\end{align}
A combination of 
\eqref{eq:hG1}, \eqref{eq:(6)a} and
\eqref{eq:(5)b} establishes \eqref{eq:h>n!1} as long as
$n_3>0$, that is, for $n\ge 8$.

\medskip
(7) This can be established in the same manner as before.
As a matter of fact, this had already been done earlier in
\cite[Theorem~1]{GrNeAB}. In fact, there the stronger result
is shown that all $s_n(\Gamma)$'s are odd. Indeed, this is
a consequence of considering the relation
\eqref{Eq:HallTransform}, viewed as a recurrence relation
for the $s_n(\Gamma)$'s, modulo~$2$.
\end{proof}

\section{The $p$-divisibility of permutation numbers}
\label{sec:perm}

It is well-known (see e.g.\ \cite[Eq.~(5.30)]{StanBI}) that
\begin{equation} \label{eq:ZykPerm} 
\sum_{n\ge0}\frac {z^n} {n!}
\sum_{\si\in S_n}
\prod _{i\ge1} ^{}x_i^{\#(\text{cycles of length $i$ in $\si$})}
=
\exp\left(\sum_{n\ge1}\frac {x_n} {n}z^n\right).
\end{equation}
In the language of species (cf.\ \cite{BLL}), this is the
{\it cycle index series} of the species of permutations.
Evidently, there are numerous specialisations of this formula
which, in combination with our $p$-divisibility results in
Sections~\ref{sec:pdiv1} and \ref{sec:pdiv2}, lead to
$p$-divisibility results for numbers of permutations with
restrictions on their cycle lengths. In this section, we present 
three prototypical such theorems.

\begin{theorem} \label{thm:perm1}
Let\/ $A$ be a subset of the positive integers, 
let $p$ be a prime number, and let $l$ be
a positive integer.
Furthermore, let $\Pi_1(n;p,l,A)$ be the number of permutations
of $\{1,2,\dots,n\}$
whose cycle lengths are in 
\begin{equation} \label{eq:PiA} 
\{ap^s:a\in A\text{ \em and }ap^s<p^l\}.
\end{equation}
Then $\Pi_1(n;p,l,A)$ is divisible by $p^{e_p(n;l)},$ where
$$
e_p(n;l)=\sum_{s=1}^{l-1}\fl{\frac {n} {p^{s}}}-(l-1)\fl{\frac {n} {p^l}}.
$$
\end{theorem}

\begin{proof}
In \eqref{eq:ZykPerm},
we set $x_i=1$ for all $i$ in the set \eqref{eq:PiA} and
$x_i=0$ for all other~$i$.
It is then straightforward to see that
all assumptions of Corollary~\ref{thm:pdiv10} are satisfied.
The assertion of the theorem then follows from \eqref{eq:vpH10}.
\end{proof}

Clearly, the special case of Theorem~\ref{thm:perm1} 
in which we consider permutations in $S_n$ 
whose cycle lengths are powers $p^s$ of
a given prime number $p$, with $0\le s\le l$, corresponds to the
enumeration of representations of $C_{p^l}$ in $S_n$. 
Thus, Theorem~\ref{thm:rang} with $r=1$ --- found earlier by
Katsurada, Takegahara and Yoshida \cite[Theorems~1.2]{KaTYAA} ---
is the simplest instance of the above theorem.

Another special case which is worth discussing explicitly is the
case of permutations in $S_n$ whose cycle lengths are strictly less than~$N$,
where $N$ is some given positive integer. 
By application of Theorem~\ref{thm:perm2}, we conclude
that the number of these permutations is divisible by
$$
\underset{p\text{ prime}}{\prod _{p\le N} ^{}}
p^{\sum_{s=1}^{\fl{\log_pN}-1}
\fl{{n} /{p^{s}}}-\left(\fl{\log_pN}-1\right)
\fl{ {n} \left/{p^{\fl{\log_pN}}}\right.}}.
$$

\begin{theorem} \label{thm:perm2}
Let\/ $A$ be a subset of the positive integers, 
let $p$ be a prime number, and let $l$ be
a positive integer.
Furthermore, let $\Pi_2(n;p,l,A)$ be the number of permutations
of $\{1,2,\dots,n\}$
whose cycle lengths are in 
\begin{equation} \label{eq:PiA2} 
\{ap^s:a\in A\text{ \em and }ap^s\le p^l\}.
\end{equation}
If $p\ge3$ and $(p,l)\ne(3,1),$ the number
$\Pi_2(n;p,l,A)$ is divisible by $p^{f_p(n;l)},$ where
$$
f_p(n;l)=
\sum_{s\ge1}\fl{\frac {n} {p^{s}}}
-(l-1)\fl{\frac {n} {p^l}}-\sum_{s \ge l}\fl{\frac {n} {2p^{s}}},
$$
while the number
$\Pi_2(n;3,1,A)$ is divisible by $3^{f_3(n)},$ where
$$
f_3(n)=
\sum_{s \ge 1}
\left(\fl{\frac {n} {3^{s}}} 
-\fl{\frac {n} {2\cdot3^{s}}} \right)
-\fl{\frac {n} {18}} .
$$
Finally, the number
$\Pi_2(n;2,l,A)$ is divisible by $2^{f_2(n)},$ where
$$
f_2(n)=
\begin{cases} 
\fl{\frac {n} {2}}
-\fl{\frac {n} {4}}
&\text{if\/ }l=1,\\[1mm]
\fl{\frac {n} {2}}
&\text{if\/ }l=2,\\[1mm]
\sum_{s=1}^{l+1}\fl{\frac {n} {2^{s}}}
-(l-1)\fl{\frac {n} {2^l}}
&\text{if\/ }l\ge3.
\end{cases}
$$
\end{theorem}

\begin{proof}
This is completely analogous to the proof of Theorem~\ref{thm:perm1}.
The only difference is that, instead of Corollary~\ref{thm:pdiv10},
one applies Theorems~\ref{thm:pdiv2}--\ref{thm:pdiv22}.
\end{proof}

\begin{theorem} \label{thm:perm3}
Let\/ $A$ be a subset of the positive integers, 
let $p$ be a prime number, and let $l$ be
a positive integer.
Furthermore, let $\Pi_3(n;p,l,A)$ be the number of permutations
of $\{1,2,\dots,n\}$
whose cycle lengths are in 
\begin{equation} \label{eq:PiA3} 
\{ap^s:a\in A\text{ \em and }ap^s< 2p^l\}.
\end{equation}
If $p\ge3$ and $(p,l)\ne(3,1),$ the number
$\Pi_3(n;p,l,A)$ is divisible by $p^{g_p(n;l)},$ where
$$
g_p(n;l)=
\sum_{s\ge1}\fl{\frac {n} {p^{s}}}
-(l-1)\cl{\frac {n} {2p^l}}-\sum_{s \ge l}\fl{\frac {n} {2p^{s}}} .
$$
\end{theorem}

\begin{proof}
This is completely analogous to the proof of Theorem~\ref{thm:perm1}.
Here, we apply Theorem~\ref{thm:pdiv3}.
\end{proof}

\section{A supercongruence}
\label{sec:super}

In the past few years it has become fashionable to call congruences
modulo prime powers $p^e$, where the exponent $e$ is at least~$2$,
{\it``supercongruences"}. The specialisation of
Corollary~\ref{thm:pdiv10} where $s_1=s_p=1$ and all other $s_n$'s
are set equal to zero leads to such a congruence.\footnote{Since
in the modulus of our congruence the exponent of $p$ even grows
with $p$, we are tempted to call this a ``supersupercongruence,"
but refrain from doing so.}

\begin{theorem} \label{eq:p-id}
For all primes $p$ and positive integers $a,b,c$
with $0\le b,c<p,$ we have
\begin{multline} \label{eq:idpm}
\sum_{s=0}^{pa+b} 
\frac {(p^2a+pb+c)!} {p^{pa+b - s}\left(pa+b - s\right)!\,(p s+c)!}\\
\equiv (-1)^ap^{(p-1)a}
\sum_{s=0}^{b} 
\frac {(pb+c)!} {p^{b - s}\left(b - s\right)!\,(p s+c)!}
\quad 
\left(\text{\em mod $p^{(p-1)a+b+1}$}\right).
\end{multline}
\end{theorem}

\begin{proof}
Let $S(z)=z+\frac {z^p}p$ and $H(z)=\sum_{n\ge0}\frac {h_n} {n!}z^n
=\exp(S(z))$. Expansion of $H(z)$, 
$$
H(z)=\exp(z)\exp\left(\frac {z^p} {p}\right)
=\sum_{s=0}^\infty \frac {z^s} {s!}
\sum_{t=0}^\infty \frac {z^{pt}} {p^t\,t!},
$$
and comparison of coefficients lead to
\begin{equation} \label{eq:hp^2N} 
h_{p^2a+pb+c}
=\sum_{s=0}^{pa+b} \frac {\big(p^2a+pb+c\big)!} 
{(ps+c)!\,p^{pa+b - s}\,\left(pa+b - s\right)!}.
\end{equation}
Moreover, by Corollary~\ref{thm:pdiv10} with $l=2$ and $n=p^2a+pb+c$,
we know that the $p$-adic valuation of 
$h_{p^2a+pb+c}$ is at least
$(p-1)a+b$, and that the quotient
$Q_{p^2a+pb+c}=h_{p^2a+pb+c}/p^{(p-1)a+b}$
satisfies
\begin{equation} \label{eq:pm}
Q_{p^2a+pb+c}\equiv (-1)^{a}Q_{pb+c} \quad (\text{mod }p).
\end{equation}
Hence, when both sides of \eqref{eq:hp^2N} are 
reduced modulo $p^{(p-1)a+b+1}$, and \eqref{eq:pm}
is used on the left-hand side, the result is \eqref{eq:idpm}.
\end{proof}

\begin{remark}
By taking more terms of the Artin--Hasse exponential 
$\sum_{n\ge0}z^{p^n}/p^n$ (cf.\ \cite[Sec.~7.2]{RobeAA} for more
information),
the above theorem can be generalised to supercongruences for
multisums.
\end{remark}

\end{document}